\documentclass{siamart}

\usepackage{graphicx}
\usepackage[bb=boondox, bbscaled=1.06,cal=euler,scr=boondoxo,scrscaled=1.04,frak=euler,frakscaled=1.05]{mathalpha}
\usepackage{algorithm}
\usepackage{algpseudocode}
\usepackage{textcomp}

\usepackage{bbm,color}
\def\C{\mathbbm{C}}
\def\R{\mathbbm{R}}
\def\Cn{\C^n}
\def\Cnn{\C^{n\times n}}
\def\CK{{\cal K}}   
\def\mrpoly{\pi}    

\def\nev{\mbox{{\it nev}}}
\def\rtol{\mbox{{\it rtol}}}

\def\pof{\mbox{{\it pof}}@}
\def\nrhs{\mbox{{\it nrhs}}}

\def\dphiin{d\phi_{\rm in}}

\def\pofcutoff{\mbox{{\it pofcutoff}}@}

\def\MaxPof{\text{{\it MaxP\kern-1pt o\kern-.25pt f}}}
\def\EatDot#1{}  
\def\wh#1{\widehat{#1}}

\def\eps{\varepsilon}

\mathcode`\@="8000
{\catcode`\@=\active \gdef@{\mkern1mu}} 

\begin{document}

\title{Polynomial Approximation to the \\ Inverse of a Large Matrix}

\author{Mark Embree\footnotemark[2]
\and Joel A. Henningsen\footnotemark[3]\and
Jordan Jackson\footnotemark[2]\and\newline\ \ \ Ronald B. Morgan\footnotemark[3] }

\maketitle

\renewcommand{\thefootnote}{\fnsymbol{footnote}}
\footnotetext[2]{Department of Mathematics, Virginia Tech, Blacksburg, VA 24061\\ ({\tt embree@vt.edu},  {\tt jgjacks4@vt.edu}).}
\footnotetext[3]{Department of Mathematics, Baylor University, Waco, TX 76798-7328\\ ({\tt joel.henningsen32@gmail.com}, {\tt Ronald\_Morgan@baylor.edu}). }

\renewcommand{\thefootnote}{\arabic{footnote}}

\begin{abstract}
The inverse of a large matrix can often be accurately approximated by a polynomial of degree significantly lower than the order of the matrix.  The iteration polynomial generated by a run of the GMRES algorithm is a good candidate, and its approximation to the inverse often seems to track the accuracy of the GMRES iteration.  We investigate the quality of this approximation through theory and experiment, noting the practical need to add copies of some polynomial terms to improve stability.  To mitigate storage and orthogonalization costs, other approaches have appeal, such as polynomial preconditioned GMRES and deflation of problematic eigenvalues.  Applications of such polynomial approximations include solving systems of linear equations with multiple right-hand sides (where the solutions to subsequent problems come simply by multiplying the polynomial against the new right-hand sides) and variance reduction in multilevel Monte Carlo methods.

\end{abstract}

\begin{keywords}
linear equations, GMRES, polynomial preconditioning, matrix inverse, eigenvalues
\end{keywords}

\begin{AMS}
65F10, 15A09
\end{AMS}

\pagestyle{myheadings}
\thispagestyle{plain}
\markboth{M. EMBREE, J. A. HENNINGSEN, J. JACKSON, AND R. B. MORGAN}{POLYNOMIAL APPROXIMATION TO AN INVERSE}

\section{Introduction}

Inverses of matrices occur throughout applications in mathematics, science, and engineering.  One rarely needs to access the inverse explicitly;  more typically one seeks the action of the inverse multiplied against a vector, i.e., the solution of a system of linear equations.  For dense matrices this solution can be found via a matrix factorization; for large, sparse problems, one typically solves the system approximately with an iterative method that, effectively, multiplies an approximation of the inverse against a vector.  Such algorithms are often motivated by invoking the Cayley--Hamilton theorem, which implies that the inverse of a matrix $A\in\C^{n\times n}$ can be expressed as a polynomial of degree $n-1$ (or less) in $A$.\ \ While it is impractical to compute this polynomial for large $n$, one can often find satisfactory approximations to $A^{-1}$ using polynomials of significantly lower degree.  We describe several ways to construct such an approximate polynomial representation of the inverse of a large matrix.  We show that this polynomial can be accurate and useful, applying it to solve systems of linear equations with multiple right-hand sides.  A second application, described in~\cite{HighDegrPolyQCD,MultPolyMonteCarlo}, uses a polynomial approximation to an inverse to reduce variance in a Multilevel Monte Carlo sampling of the trace of the inverse.  Given the ubiquity of the matrix inverse, we anticipate numerous other applications.

This work builds on the development of stable polynomial preconditioners for eigenvalue problems~\cite{PPArn,Tho06} and linear systems~\cite{PPGStable} based on the GMRES (Generalized Minimum Residual) algorithm~\cite{SaSc}.  
Here, we focus on the quality of a polynomial $p(A)$ as an approximation of $A^{-1}$ itself,
addressing theoretical aspects but focusing on practical considerations required to make such
approximations useful.
An implementation of the GMRES residual polynomial using its roots (harmonic Ritz values) is described in~\cite{PPArn}.  This approach is more stable than most previous methods for implementing the GMRES polynomial~\cite{seed,Jo,PPG,NaReTr} and cheaper to implement than another approach~\cite{Tho06}.  However, this polynomial can be prone to instabilities, prompting the stability control method proposed in~\cite{PPArn}: the polynomial is augmented with extra copies of roots near outstanding eigenvalues.  This modification enables the practical use of higher degree polynomials.

When GMRES is applied to $Ax=b$, the residual vector can be written as 
\[r = b - A@\wh{x} = b - A@p(A)b = \pi(A)b,\] 
where $\wh{x} = p(A)b$ is the approximate solution generated by GMRES, and $\pi(z) = 1 - z p(z)$ is the GMRES residual polynomial.  For $r$ to be small in norm, $p(A)$ should approximate $A^{-1}$ (modulated by the vector $b$), and this $p$ provides the starting point for our approximate inverse polynomial.  In~\cite{PPGStable}, an algorithm is given for multiplying $p(A)$ against a vector, given the roots of $\pi$.  This implementation uses the technique from~\cite{PPArn} for adding extra roots to $\pi$ to promote stability, often making it possible to find a moderate degree polynomial $p$ that achieves the goal of $p(A) \approx A^{-1}$.  

We show that the accuracy of the polynomial $p(A)$ as an approximation to $A^{-1}$ follows the GMRES residual.  To reduce the cost of orthogonalization, we explore the use of a composite (or double) polynomial generated from polynomial preconditioned GMRES~\cite{PPGStable}, along with the nonsymmetric Lanczos algorithm. 
For systems with multiple right-hand sides, once a polynomial has been found, the solution of additional systems simply requires multiplying $p(A)$ against the new right-hand sides.  We also study a version of the polynomial that incorporates deflation, and can be effective at a lower degree.

The rest of this section reviews some previous work that is needed for this project.  New material begins in Section~\ref{sec:poly} with use of GMRES to find a polynomial approximation to the inverse, and provides some theoretical results on the quality of this approximation.  We also review harmonic Ritz values and give a new bound on their location.  Section~\ref{sec:alt} investigates other approaches for constructing the polynomial: restarted GMRES, nonsymmetric Lanczos, and polynomial preconditioned GMRES.\ \  Section~\ref{sec:mult_rhs} uses the polynomial to solve systems with multiple right-hand sides, then  Section~\ref{sec:deflated} describes how to incorporate deflation.  Finally, Section~\ref{sec:stab} considers some examples that provide challenges for stabilization.


\subsection{Polynomial preconditioning} \label{ssec:ppgmres}

When solving $Ax = b$, polynomial preconditioning is a way to transform the spectrum and thus improve convergence.
With a polynomial $p$ and right preconditioning, the linear system becomes 
\begin{equation}
Ap(A)y = b, \qquad x = p(A)y. \label{eqn:ppsys}
\end{equation}
Defining $\phi(z) \equiv z@ p(z)=1-\pi(z)$, the preconditioned system of linear equations is 
$\phi(A)y = b.$  
We let $dp$ denote the degree of $p$, so $\phi$ has degree $d\phi \equiv dp+1$.  

Much work has studied polynomial preconditioning; see, e.g., \cite{seed,As87,AsMaOt,PPArn,FiRe,Jo,La52B,LiXiVeYaSa,PPG,PPGStable,Rutis,Sa84b,Sa87b,Sa03,SmSa,Sti58,Tho06,vGi95,YeXiSa}.  We highlight Thornquist's thesis~\cite{Tho06}, which constructs polynomial approximations to $(A-\mu B)^{-1}$ using GMRES (and several nonsymmetric Lanczos methods) to expedite shift-invert eigenvalue calculations.  The use in~\cite{Tho06} of non-optimal short-recurrence Krylov subspace methods to generate polynomial approximations to matrix inverses merits further investigation.

In~\cite{PPArn,PPGStable}, starting with the GMRES residual polynomial $\pi$, the polynomial $\phi$ is chosen as $\phi(z) = 1 - \pi(z)$ and thus $p$ is also determined.  The roots of $\pi$ are the harmonic Ritz values~\cite{IE,IEN,PaPavdV}, and they are used to implement both polynomials $\phi$ and $p$.  
The paper~\cite{PPGStable} includes detailed algorithms that use the roots of $\pi$ to apply $\phi(A)$ and $p(A)$ to vectors (respecting complex conjugate pairs of roots); see \cite[alg.~1 and alg.~3]{PPGStable}.  Thus the polynomials needed for polynomial preconditioning can be determined with one cycle of GMRES (frequently with a random starting vector~\cite{PPGStable}).  Then GMRES can also be used to solve the linear equations as a part of polynomial preconditioned GMRES~\cite[alg.~4]{PPGStable}.  

\subsection{Stability for the polynomial}  \label{ssec:rev2}

The GMRES residual polynomial $\pi(z)$ is a product of terms of the form $(1-z/\theta_k)$, where $\theta_k$ is a \emph{harmonic Ritz value} (see section~\ref{sec:hritz}).
If $A$ has an eigenvalue that stands out from the rest of the spectrum, GMRES typically places a single $\theta_k$ nearby, giving $\pi$ a steep slope at that root that can lead to ill-conditioning and cause $p(A)$ evaluations to be unstable.  To improve stability, extra copies of roots corresponding to outstanding eigenvalues can be added.
This is implemented in~\cite[p.~A21]{PPArn} (see also~\cite[alg.~2]{PPGStable}).  For each root $\theta_k$, one computes a diagnostic quantity called $\pof(k)$ (for ``product of other factors'') that measures the magnitude of $\pi(\theta_k)$ with the $(1-z/\theta_k)$ term removed.  When $\log_{10}(\pof(k))$ exceeds some threshold, which here we call $\pofcutoff$, extra $(1-z/\theta_k)$ terms are appended to $\pi$.  In this paper, we use $\pofcutoff=8$, unless it is specified otherwise. 

By construction, $p(z)$ interpolates $1/z$ at the roots $\theta_k$ of $\pi$.
When we add a second copy of a root to stabilize $\pi(z)$, that root becomes a point where $p'(z)$ interpolates $(1/z)' = -1/z^2$.  If $\theta$ is an $m$-fold root of $\pi$,  one can show  
$p^{(j)}(\theta) = (-1)^j j!/\theta^{j+1} = {\rm d}^j/{\rm d}z^j\,(1/z) \big|_{z=\theta}$ for $j=1,\ldots, m-1$: thus $p(z)$ interpolates $1/z$ and its first $m-1$ derivatives at $z=\theta$.  
Figure~\ref{fig:cartoon} shows how a second copy of a root can stabilize $\pi(z)$ and $p(z)$ in the proximity of an outlying eigenvalue.

\begin{figure}[b!]
\begin{center}
\includegraphics[width=2.4in]{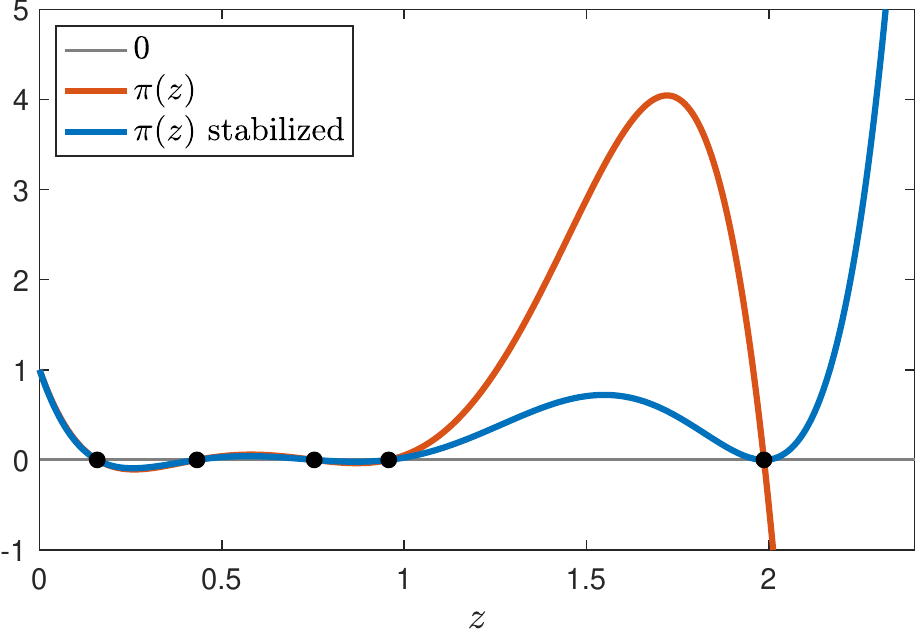} \qquad
\includegraphics[width=2.4in]{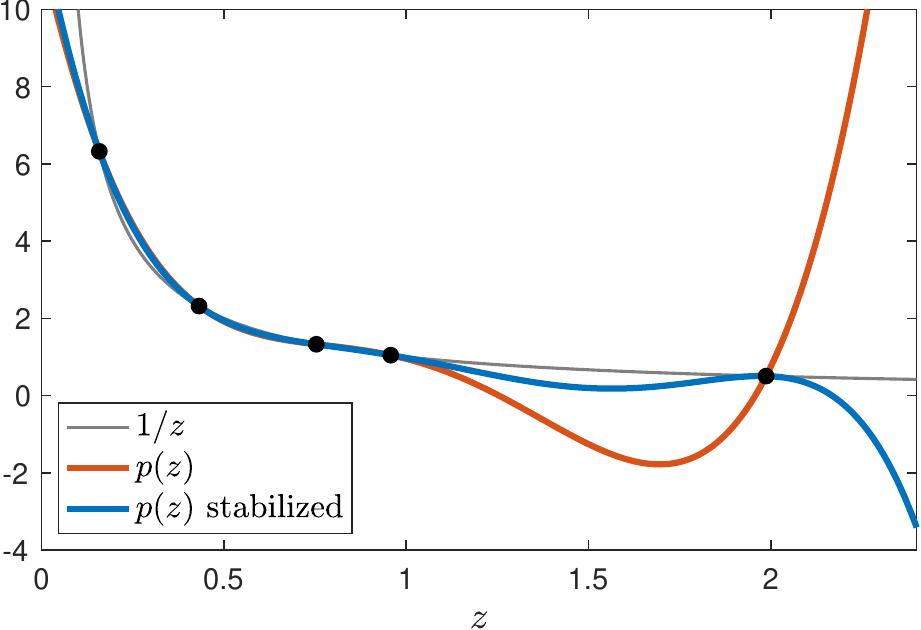} 
\end{center}

\vspace*{-7pt}
\caption{\label{fig:cartoon}
Adding an extra copy of a root can stabilize $\pi$ and $p$.
This $A$ has real eigenvalues in $[0.1,1]$ and one outlying eigenvalue at $\lambda=2$.
Left: the degree $k=5$ GMRES residual polynomial $\pi(z)$ (red) with roots at the black dots; note the steep slope at the root $\theta = 1.9869$ near the eigenvalue $\lambda=2$.  Adding an extra copy of that root leads to the degree $k+1=6$ polynomial (blue) that is small in a larger neighborhood of the root.
Right: the corresponding approximate inverse polynomial $p(z)$, which interpolates $1/z$ (gray line) at the black dots.  The extra root has a similarly tonic effect on this polynomial, which now also interpolates the derivative of $1/\lambda$ at the extra root.
}
\end{figure}

\subsection{Deflation} \label{sec:reviewsub3}

When solving linear systems, \emph{deflation} refers to reducing the influence of small eigenvalues that tend to slow GMRES convergence. Deflation can be implemented by adding approximate eigenvectors to a subspace~\cite{GMRES-E,GMRES-DR}, or by building a preconditioner from eigenvectors; see, e.g., \cite{KhYe,PadeStMaJoMa,SiEmMo}.  Here we will deflate with Galerkin projection~\cite{MultPolyMonteCarlo,GMRES-DR,gproj,SiGa,StOr} (alg.~2 in \cite{MGLE}).  This method can be applied before running a Krylov method, or between cycles of restarted GMRES.\ \ Section~\ref{sec:deflated} has a new use of deflation.

\section{Approximating \boldmath $A^{-1}$ with a polynomial in $A$} \label{sec:poly}

\subsection{Using GMRES to find an approximating polynomial}

Let $\pi$ denote a GMRES residual polynomial
formed after sufficient iterations 
to solve the linear system $Ax=b$ to desired tolerance, and define $p$ via $\pi(z) = 1- z@p(z)$.  
Examples will show that it is possible for $p(A)$ to be a good approximation to $A^{-1}$.  The accuracy of this approximation typically tracks the GMRES residual norm; see Theorem~\ref{thm:normal}.

When a standard preconditioner is available for $A$, we can still get an approximation to $A^{-1}$, but it will not be a polynomial of  $A$ alone.
Looking at the system with standard right preconditioning, where $M$ is an approximation to $A$,
\[ A M^{-1} w = b,  \qquad
\widehat{x} = M^{-1}w, 
\]
we see $r = b-A@\widehat{x} = b - A M^{-1} p(A M^{-1})b = \big(I-A M^{-1} p(A M^{-1})\big)b$, where $p$ comes from GMRES applied to the preconditioned system.  For $\|r\|$ to be a small, $A^{-1}$ should be approximated by $M^{-1}p(A M^{-1})$.  See Examples 6 and 7 for use of this preconditioned polynomial.  

Our first example shows how polynomial approximation of the inverse can work with a matrix from a simple application problem.

{\it Example 1.} 
Consider the standard second-order finite difference discretization of the convection-diffusion equation $- u_{xx} - u_{yy} + 2 u_{x} = f $ on a uniform grid over the unit square $[0,1]\times[0,1]$, with homogeneous Dirichlet boundary conditions.  Grid size $1/50$ gives $A$ of order $n = 2500$.  GMRES is run with a random starting vector (normal entries, scaled to unit norm).  Figure~\ref{fig:polyacc1} compares GMRES convergence to the relative accuracy of the polynomial, $\|A^{-1} - p(A)\| / \|A^{-1}\|$.  Notice that $p(A)$ can accurately approximate $A^{-1}$;  the relative accuracy of $p(A)$ keeps pace with the GMRES residual norm, typically about an order of magnitude behind.  The GMRES residual norm goes below $10^{-12}$ at iteration $k=217$, where $p(A)$ has relative accuracy of $5.1\times 10^{-12}$.

\begin{figure}[b!]
\begin{center}
\includegraphics[width=3.225in]{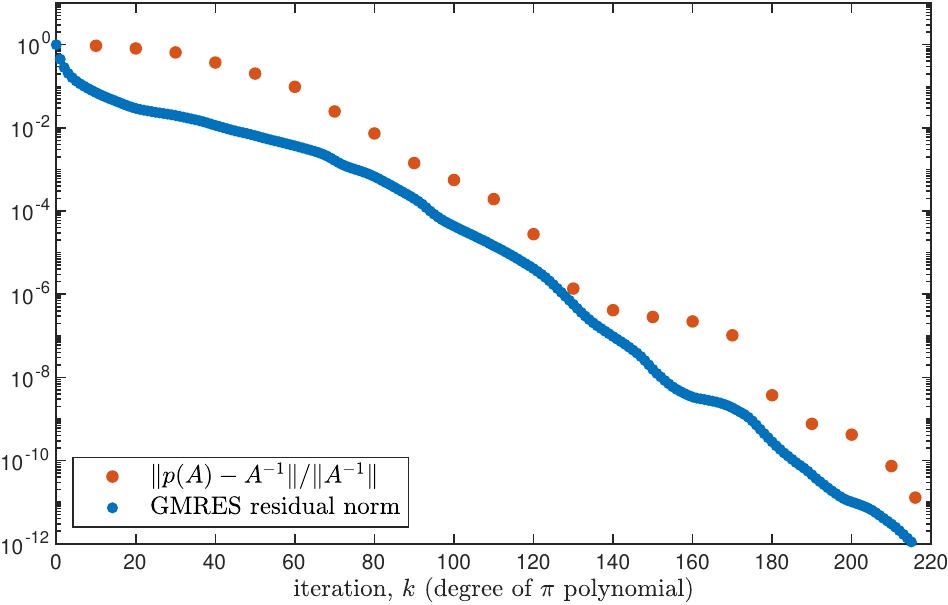}
\end{center}

\vspace*{-6pt}
\caption{\label{fig:polyacc1}
Example 1:  The matrix is size $n = 2500$ from the convection-diffusion equation $- u_{xx} - u_{yy} + 2 u_{x} = f $.  The relative accuracy, $\|A^{-1} - p(A)\| / \|A^{-1}\|$, of the polynomial approximation to the inverse is shown every 10 GMRES iterations, and is compared to the GMRES residual norm.}
\end{figure}

\begin{figure}[t!]
\begin{center}
\includegraphics[width=4.3in]{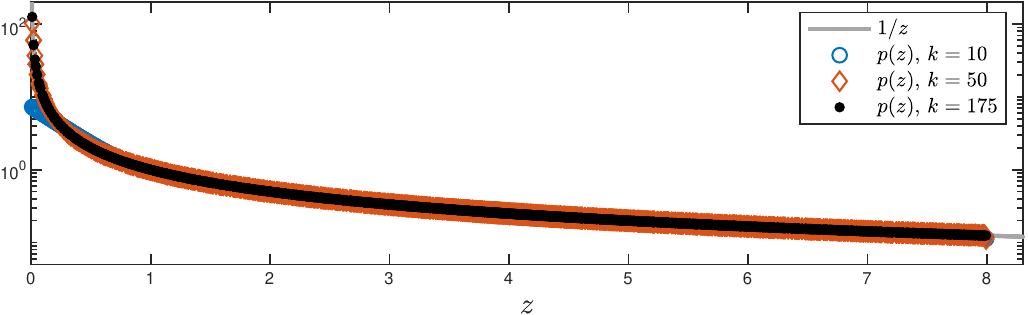}\\[7pt]
\includegraphics[width=4.3in]{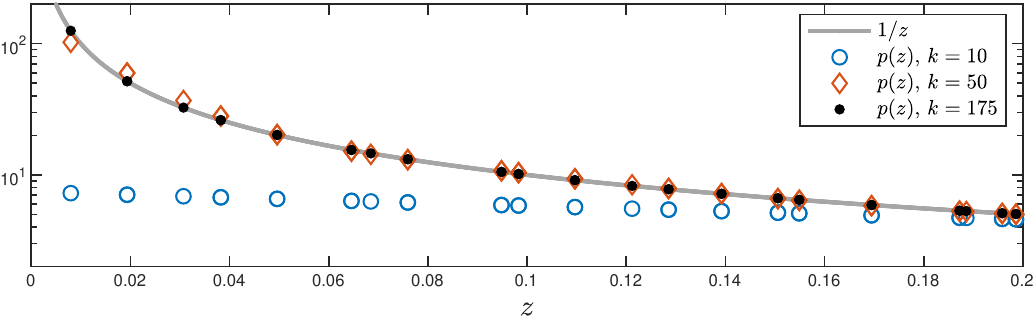}
\end{center}
\begin{picture}(0,0)
\put(190,90){\small\emph{Zoom on left end}}
\put(190,80){\small\emph{of the spectrum}}
\end{picture}
\vspace*{-12pt}
\caption{Convection-diffusion example from Figure~\ref{fig:polyacc1} (dimension $n=2500$), showing how the polynomial $p(z)$ approximates $1/z$ at the eigenvalues of $A$: the gray line shows $1/z$; the markers show $p(\lambda)$ for the eigenvalues $\lambda$ of $A$, for three different degree polynomials $p$ of degree $k-1$.}
\label{fig:polyandrecip}
\end{figure}
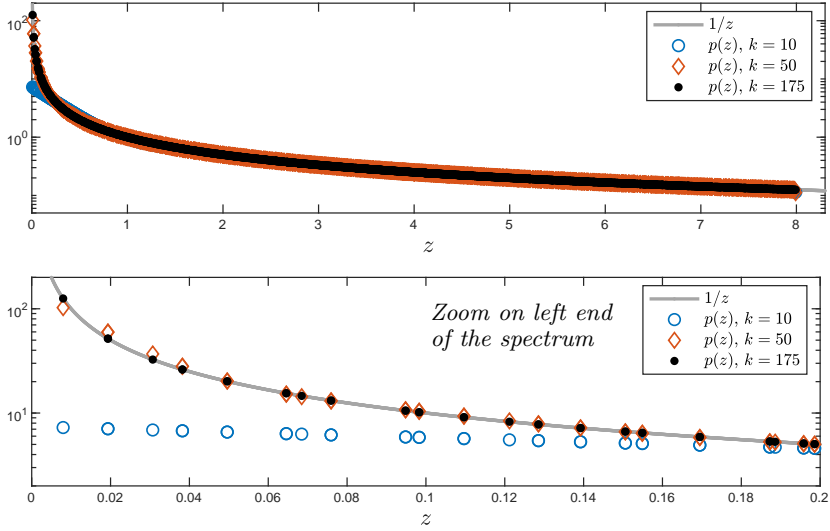

\begin{figure}[h!]
\begin{center}
\includegraphics[width=2.35in]{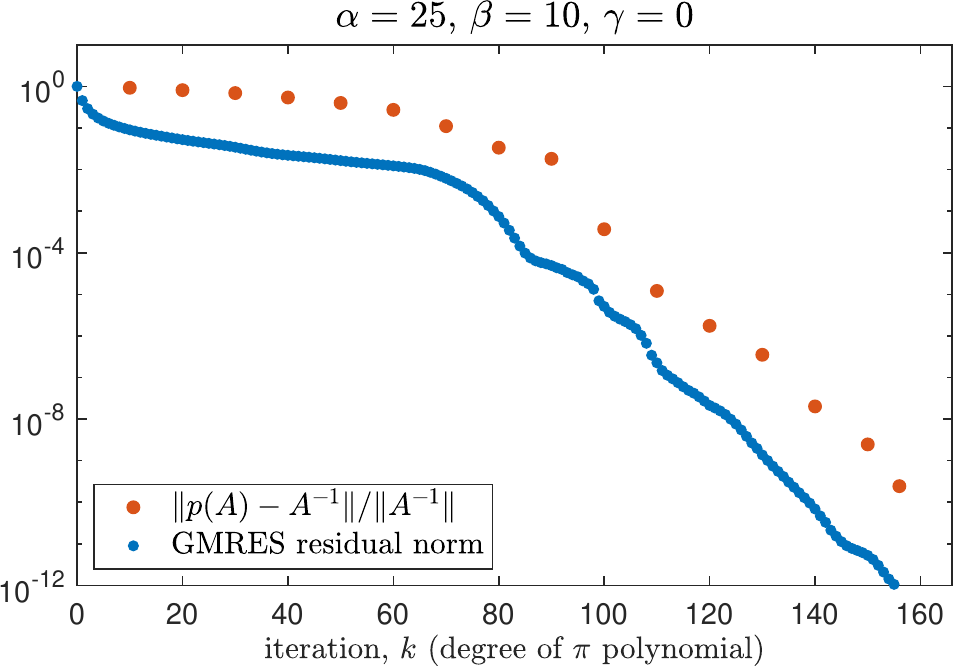}\quad
\includegraphics[width=2.35in]{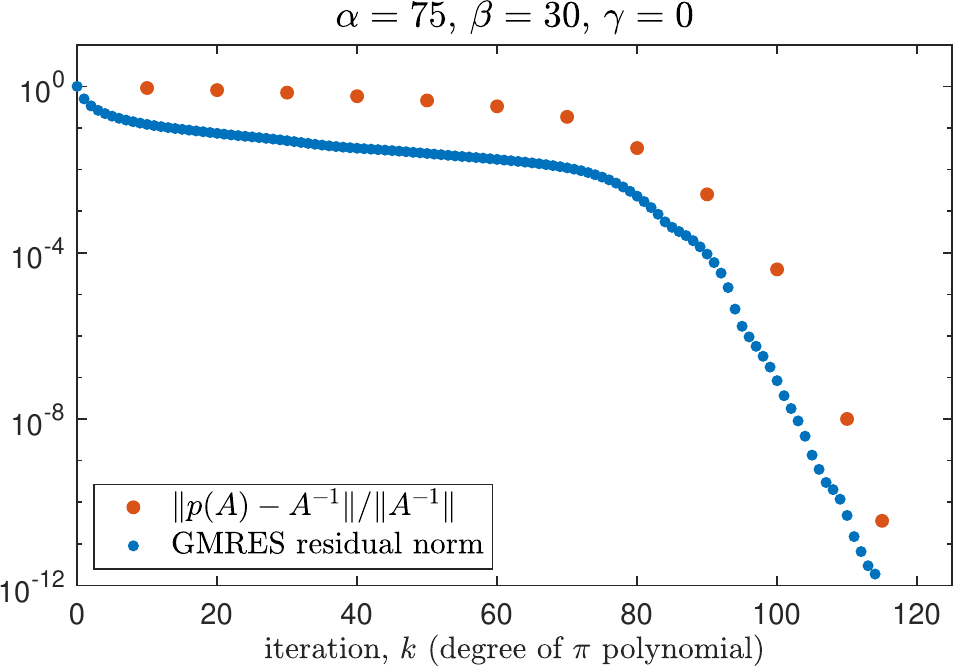}

\vspace*{8pt}
\includegraphics[width=2.35in]{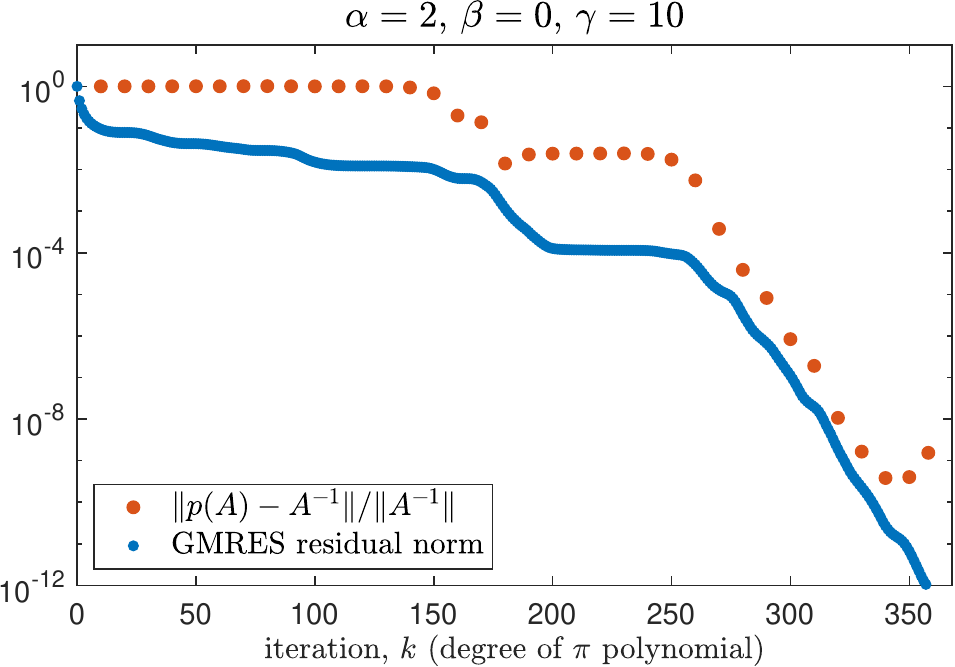}\quad
\includegraphics[width=2.35in]{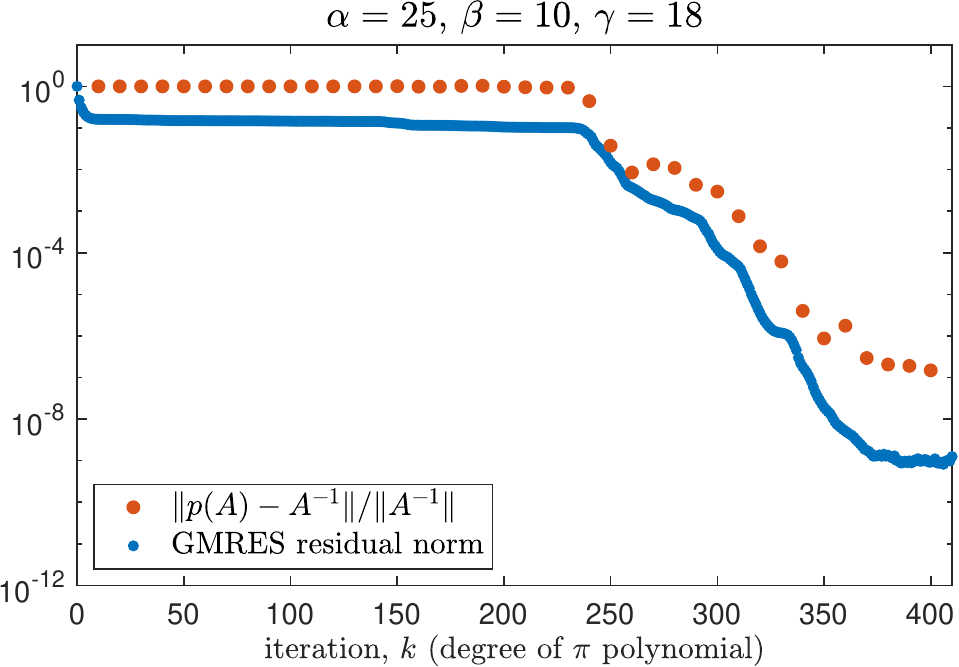}
\end{center}

\vspace*{-5pt}
\caption{Matrix of size $n = 2500$ from the convection-diffusion equation $- u_{xx} - u_{yy} + \alpha u_{x} + \beta u_{y} - \gamma^2 u = f $.  The relative accuracy of $p(A)$ is compared to the GMRES residual for variying degrees of nonnormality and indefiniteness.  (The last few red dots in the bottom plots likely reflect some numerical instabilities.)}
\label{fig:polyacc2}
\end{figure}


Figure \ref{fig:polyandrecip} compares the polynomial $p(z)$  to $f(z) = 1/z$.  For $p(A)$ to approximate $A^{-1}$ well, $p(z)$ must be close to $1/z$ at every eigenvalue of $A$ (made precise later in this section).  At iteration $k=10$, $p$ looks to be a good approximation over the entire spectrum of $A$ (top plot).  However, near the origin (bottom plot) we see that the $k=10$ polynomial does not approximate $1/z$ well at the small eigenvalues.  The $p$ at iteration $k=50$ is better, but still off a bit.  The $p$ at $k=175$ (computed at the point where the GMRES residual hits $10^{-9}$) is accurate at all eigenvalues.

In Figure~\ref{fig:polyacc2} we adjust parameters in the differential equation to make $A$ more nonsymmetric, then also indefinite.  For  $- u_{xx} - u_{yy} + \alpha u_{x} + \beta u_{y} - \gamma^2 u = f $, we first let $\alpha = 25$, $\beta=10$ and $\gamma=0$.  The accuracy with which $p(A)$ approximates $A^{-1}$ (upper-left) follows the GMRES residual norm, though not quite as closely as in Figure~\ref{fig:polyacc1} with a nearly symmetric $A$.\ \ The polynomial remains effective with more nonnormality (upper-right) and with indefiniteness (lower-left).  However, the accuracy is more erratic when the matrix is both significantly nonnormal and indefinite (lower-right).

\subsection{Polynomial approximation accuracy and the GMRES residual}

This first result shows that, for well-conditioned problems, the norm of the GMRES residual can serve as an indicator for the relative error in the polynomial approximation of the inverse:  If the norm of the residual remains large, we cannot expect the corresponding polynomial to approximate the inverse with any accuracy.  

\begin{theorem} \label{thm:general}
Suppose $A\in\Cnn$ is invertible and $Ax=b$ for some unit vector $b\in\Cn$.
Let $\wh{x} = p(A) b$ be an approximation to $x$ for some polynomial $p$, and let $\kappa(A) = \|A\| \|A^{-1}\|$.
Then 
\begin{equation} \label{eq:general}
\frac{\|r\|}{\kappa(A)} 
   \ \le\ \frac{\|x-\wh{x}\|}{\|A^{-1}\|} 
   \ \le\ \frac{\|A^{-1} - p(A)\|}{ \|A^{-1}\|}.
\end{equation}
\end{theorem}
\begin{proof}
Since $r= b - A \wh{x} = b - A@ p(A) b = A(x-\wh{x})$, we can write
\[ x - \wh{x} \ =\ A^{-1}r\ =\ \big(A^{-1} - p(A)\big)b.\]
Since $\|b\|=1$, we have $\|x-\wh{x}\| \le \|A^{-1}-p(A)\|$.
Divide by $\|A^{-1}\|$ to get the second inequality in~\eqref{eq:general}.  The first follows from 
$\|r\| = \|A(x-\wh{x})\| \le \|A\| \|x-\wh{x}\|$.
\end{proof}

We can also use the residual to get an \emph{upper bound} on the error of the polynomial approximation.  

\begin{theorem} \label{thm:normal}
Let $A\in\Cnn$ be an invertible diagonalizable matrix with $Ax=b$ for some unit vector $b\in\Cn$.
Suppose $\wh{x} = p(A) b$ for some polynomial $p$,
and let $A = Z\Lambda Z^{-1}$ be a diagonalization of $A$ with eigenvectors $z_1, \ldots, z_n$ and corresponding eigenvalues $\lambda_1, \ldots, \lambda_n$.  
Expand $b = Z(Z^{-1}b) = \sum \widehat{\beta}_i z_i$, and define $\beta_i := \widehat{\beta}_i/\|Z^{-1}\|$.  Then for $\kappa(Z) = \|Z\| \|Z^{-1}\|$, we have
\begin{equation} \label{eq:normal}
    \frac{\|A^{-1}-p(A)\|}{\|A^{-1}\|}\ \leq\ \kappa(Z)\, \frac{\|r\|}{\min|\beta_i|}.
\end{equation}
\end{theorem}
\begin{proof}
Again writing $r = b - A \wh{x} = b - A@p(A)b$, we have
\begin{align*}
    \|A^{-1}\|\|r\|  \geq \|A^{-1}r\| &\ =\ \|(A^{-1}-p(A))b\| \\
    & \ =\  \|Z(\Lambda^{-1} - p(\Lambda)) Z^{-1} b \| \\
    & \ \ge \frac{1}{\|Z^{-1}\|} \|(\Lambda^{-1} - p(\Lambda)) Z^{-1}b \|
    =
    \|(\Lambda^{-1} - p(\Lambda)) [\beta_1\ \cdots\  \beta_n]^T \|.
\end{align*}
Since $\Lambda^{-1}-p(\Lambda)$ is diagonal, we have 
\begin{align*}
 \|(\Lambda^{-1} - p(\Lambda)) [\beta_1\ \cdots\  \beta_n]^T \| \ \ge\  
    \|\Lambda^{-1} - p(\Lambda)\| \min_i |\beta_i|.
\end{align*}
Now since
\[ \|A^{-1} - p(A)\| = \|Z(\Lambda^{-1} - p(\Lambda)) Z^{-1}\| \le \|Z\| \|Z^{-1}\| \|\Lambda^{-1} - p(\Lambda)\|,\]
we can further bound
\[\|\Lambda^{-1} - p(\Lambda)\| \min_i |\beta_i|
    \ge \frac{\|A^{-1} - p(A)\|}{\|Z\| \|Z^{-1}\|} \min_i |\beta_i|.
\]
In summary, we have shown that
\[ \|A^{-1}\|\,\|r\| \ge \frac{\|A^{-1} - p(A)\|}{\|Z\| \|Z^{-1}\|} \min_i |\beta_i|,\]
which can be rearranged to give~\cref{eq:normal}.
\end{proof}

By the bound~\cref{eq:normal}, when $A$ is not far from normal (i.e., when $\|Z\| \|Z^{-1}\|$ is not too large), a small GMRES residual norm implies a small relative error in the $A^{-1}$ approximation, \emph{provided $b$ is not deficient in any eigenvectors.}  To make such deficiencies unlikely, one can take $b$ to have normally distributed random entries.\footnote{For the sake of illustration, suppose $A$ is real symmetric, so $Z\in\R^{n\times n}$ can be taken to be a real unitary matrix.  Thus if $b\in\R^n$ has random Gaussian entries, $Z^*b$ will also have random Gaussian entries~\cite[thm.~1.2.6]{Mui82}.  In this case we can explicitly calculate that the probability of $\min_i |\beta_i| \le \varepsilon$ for any $\varepsilon \in (0,1)$ is $1 - {\rm erfc}(\varepsilon/\sqrt{2})^n$, where ${\rm erfc}(\cdot)$ denotes the complementary error function.  For example, when $\varepsilon=10^{-6}$ and $n=\mbox{10,000}$, this probability is less than $1\%$.}

\subsection{Bounding \boldmath $\|A^{-1} - p(A)\|$}

The quality with which $p(A)$ approximates $A^{-1}$ depends on how well $p(\lambda)$ approximates $1/\lambda$ on the spectrum of $A$, denoted by $\sigma(A)$, or some larger set depending on the departure of $A$ from normality.  
Consider the \emph{numerical range} of $A$, which is the set of all Rayleigh quotients~\cite[chap.~1]{HJ91},
\[ W(A) = \Big\{ \frac{v^* A v}{ v^*v} : 0 \ne v \in \C^n\Big\},\]
and the $\eps$-pseudospectrum of $A$~\cite{TrEm},
\[ \sigma_\eps(A) = \Big\{ z \in \C: \|(zI-A)^{-1}\| > 1/\eps\Big\}.
\]

The following proposition follows from conventional bounds on functions of matrices, and all have analogs for bounding the residual polynomial in GMRES~\cite{Emb22}.

\begin{proposition} \label{prop:approx_inv}
Suppose $A\in\C^{n\times n}$ is invertible, and let $p$ be any polynomial.
\begin{enumerate}
\item  If $A$ is diagonalizable, $A=Z \Lambda Z^{-1}$, then with $\kappa(Z) = \|Z\| \|Z^{-1}\|$, 
\[ \|A^{-1} - p(A)\| \le \kappa(Z) \max_{\lambda \in \sigma(A)} \Big|\frac{1}{\lambda} - p(\lambda)\Big|.
\]
\item If $0 \not \in W(A)$, then, by the Crouzeix--Palencia theorem~\cite{CP17},
\[ \|A^{-1} - p(A)\| \le (1+\sqrt{2}) \max_{\lambda \in W(A)} \Big|\frac{1}{\lambda} - p(\lambda)\Big|.
\]
\item If $\eps>0$ is sufficiently small that $0 \not \in \sigma_\eps(A)$, then
\[ \|A^{-1} - p(A)\| \le \frac{L_\eps}{2\pi\eps} \max_{\lambda \in \sigma_\eps(A)} \Big|\frac{1}{\lambda} - p(\lambda)\Big|,
\]
where $L_\eps$ denotes the length of the boundary of $\sigma_\eps(A)$.
\end{enumerate}
\end{proposition}


\subsection{Harmonic Ritz values} \label{sec:hritz}
The roots of $\pi(z) = 1 - z p(z)$ are precisely the points where $p(z)$ interpolates $1/z$.  Thus in light of Proposition~\ref{prop:approx_inv}, to understand how well $p(A)$ approximates $A^{-1}$ one naturally asks, \emph{Where are the roots of $\pi$ located?} 

At step $k$, GMRES computes the approximation $\wh{x}$ from the Krylov subspace 
\[\CK_{k}(A,b) 
    \ =\ {\rm span}\{b, Ab, \ldots, A^{k-1}b\}
    \ =\ \{ \psi(A)@b: {\rm deg}(\psi) < k\}.
\]
Conventional implementations of GMRES progressively
build an orthonormal basis $\{v_1, \ldots, v_k\}$
for $\CK_k(A,b)$.
This process is compactly summarized in the expression
\begin{equation} \label{eq:arnoldi}
A V_k = V_{k+1} H_{k+1,k} = V_k H_{k,k} + h_{k+1,k} v_{k+1} e_k^*,
\end{equation}
where the $(k+1)\times k$ upper Hessenberg matrix $H_{k+1,k}$ 
(and its top $k\times k$ block, $H_{k,k}$) collects the coefficients from the orthogonalization process.  Premultiplying~\eqref{eq:arnoldi} by $V_k^*$ gives $H_{k,k} = V_k^*AV_k^{}$.
(The eigenvalues of $H_{k,k} \in \C^{k\times k}$ are called \emph{Ritz values};
they are Rayleigh--Ritz eigenvalue estimates for $A$ from the Krylov subspace 
$\CK_k(A,b)$.)

The GMRES approximation $\wh{x}\in\CK_k(A,b)$ is selected to minimize the 2-norm of the residual, $\|r\| = \|b - A \wh{x}\|$, and so basic least squares theory implies that the residual $r = b - A\wh{x}$ is orthogonal to the space $A \CK_k(A,b)$ from which $b$ is approximated.  Using this property, one can show~\cite{Cao97,MO94} that each root $\theta$ of $\pi$ must satisfy the generalized eigenvalue problem 
\begin{equation} \label{eq:hritz_gep0}
H_{k+1,k}^*H_{k+1,k}^{}\,c = \theta H_{k,k}^* c, \rlap{$\qquad$ \mbox{for nonzero $c\in\C^k$.}}
\end{equation}

Since the coefficient matrices $H_{k+1,k}^*H_{k+1,k}^{}$ and $H_{k,k}^*$ are $k\times k$, this generalized eigenvalue problem can have up to $k$ finite eigenvalues, the roots of $\pi$.  When $H_{k,k}$ is singular the generalized eigenvalue problem can have fewer than $k$ finite eigenvalues; this case corresponds to the stagnation of GMRES at step $k$. 
When $H_{k,k}$ is invertible, (\ref{eq:hritz_gep0}) can be reduced to the standard eigenvalue problem
\begin{equation} \label{eq:hritz_gep}
\Big(H_{k,k} + |h_{k+1,k}|^2 f_k^{} e_k^*\Big) c = \theta c,
\rlap{\qquad $f_k := H_{k,k}^{-*} e_k$.}
\end{equation}
The roots of $\pi$, characterized by~\eqref{eq:hritz_gep0}, are called \emph{harmonic Ritz values}~\cite{Fr92,IE,IEN,PaPavdV}, since \emph{$1/\theta$ is a Rayleigh--Ritz eigenvalue estimate for $A^{-1}$ from the subspace $A\CK_k(A,b) = {\rm range}(AV_k)$}.
Indeed, one can show that
\[ (AV_k)^* A^{-1} (A V_k) c = \frac{\,1\,}{\theta} (AV_k)^*(AV_k) c,\]
 a generalized Rayleigh quotient for $A^{-1}$:
Unlike standard Ritz values (the eigenvalues of $H_{k,k}$), \emph{harmonic Ritz values are not shift invariant}: their location depends on the proximity of the spectrum of $A$ to the origin.
Since the reciprocals of the harmonic Ritz values $\{\theta_j\}_{j=1}^k$ are Ritz values for $A^{-1}$, we have $1/\theta_j \in W(A^{-1})$, that is, 
\[ \theta_j \in \{ 1/z: z \in W(A^{-1})\} 
   = \Big\{ \frac{v^*v}{v^*A^{-1} v} : 0 \ne v \in \Cn\Big\}.\]
Look ahead to Figure~\ref{fig:rvhrvex} to see an example of $1/W(A^{-1})$.
This set contains $z=\infty$ when $0 \in W(A^{-1})$, which implies $0\in W(A)$ and thus the ability of the first step of GMRES to stagnate.\ \ 
(The potential for stagnation at later iterations can be described via higher-dimensional generalizations of the numerical range~\cite[Theorem~2.7]{FJKM96}.)

\begin{figure}[t!]
\begin{center}
\includegraphics[width=2.4in]{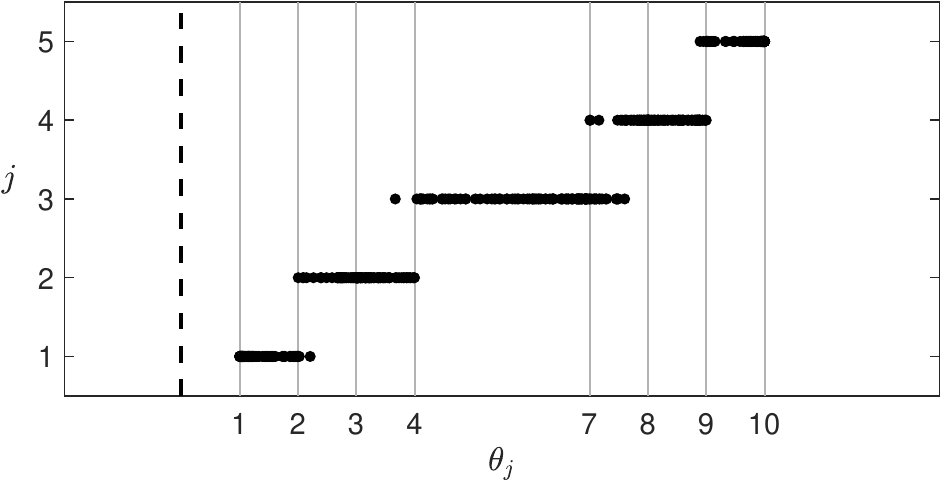}\quad
\includegraphics[width=2.4in]{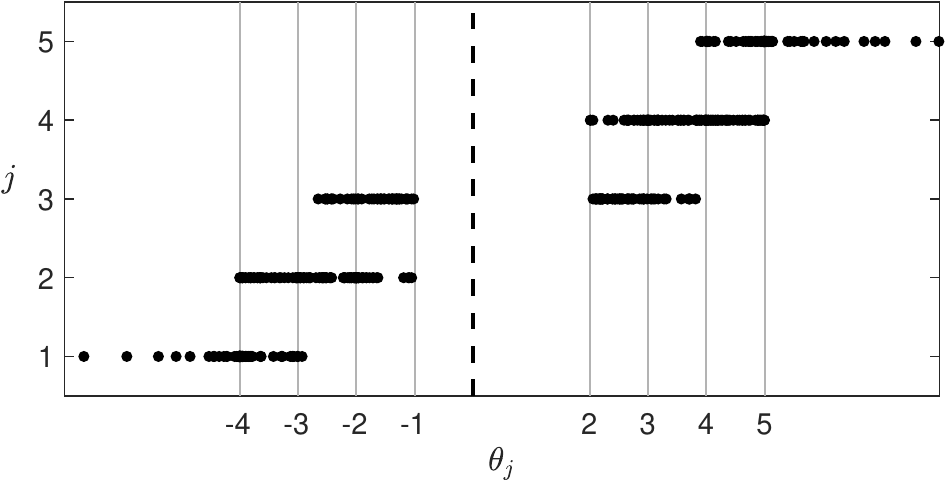}

\vspace*{3pt}
\hspace*{2pt}
\includegraphics[width=2.36in]{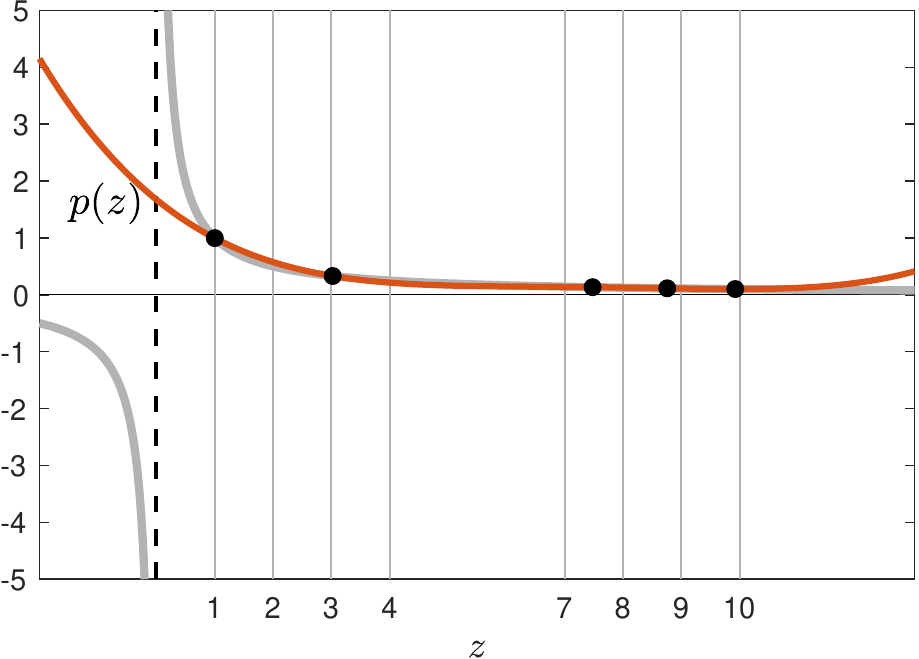}\hspace{9pt}
\includegraphics[width=2.36in]{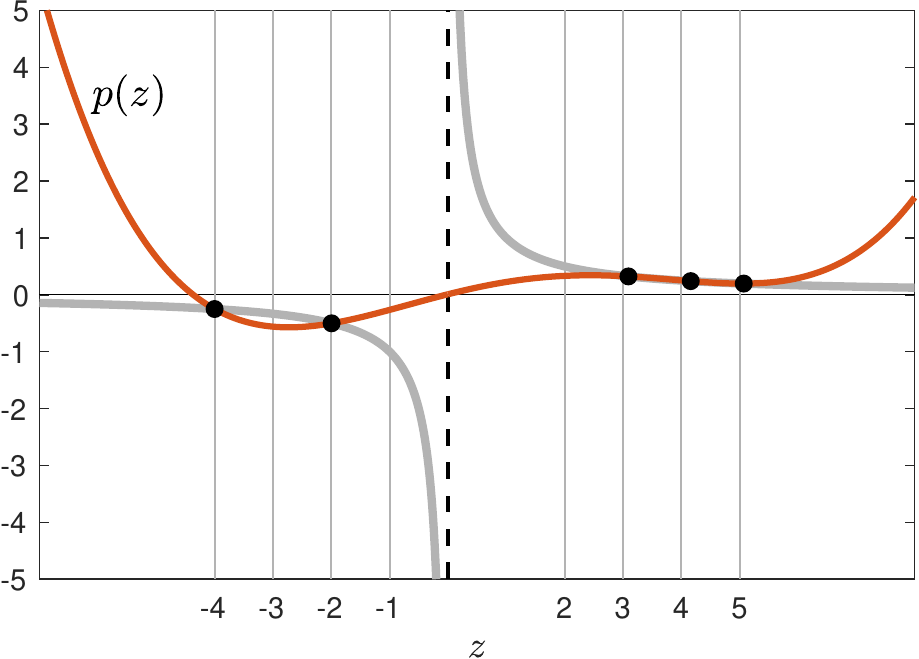}
\end{center}

\vspace{-6pt}
\caption{\label{fig:hritz_plots}
The top plots show the harmonic Ritz values from 100 trials involving random Krylov subspaces of dimension $k=5$ for Hermitian matrices of dimension $n=8$, sorted $\theta_1 \le \theta_2 \le \cdots \le \theta_5$.  Eigenvalues are marked by gray vertical lines; the origin is denoted by the vertical black dashed line.  The positive definite matrix on the left has eigenvalues $\{1,2,3,4,7,8,9,10\}$, and the harmonic Ritz values obey Cauchy interlacing.  The indefinite problem on the right has the same eigenvalues, only shifted left: $\{-4,-3,-2,-1, 2,3,4,5\}$.  Notice the absence of harmonic Ritz values in the interval $(-1,2)$; numerous $\theta_j$ values are beyond the axis limits for this case. The two plots would be identical (up to the shift) if they showed standard Ritz values, which are shift invariant.  The bottom plots show how $p(z)$ (red line) interpolates $1/z$ (gray line) for one of the trials.  Notice that approximating $1/z$ at eigenvalues on both sides of the origin (indefinite problem on the right) is much more challenging.}
\end{figure}

The case of Hermitian $A$ already shows the subtle nature of harmonic Ritz values.  If $A$ is positive definite with eigenvalues $0 < \lambda_1 \le \lambda_2 \le \cdots \le \lambda_n$,  the harmonic Ritz values obey Cauchy interlacing~(see, e.g., \cite{Par98}):  
if $\theta_1 \le \theta_2 \le \cdots \le \theta_k$, then 
\[ \lambda_j \le \theta_j \le \lambda_{n-j+k}, \rlap{\qquad $j=1,\ldots,k$.}\]
In contrast, if $A$ is indefinite (and thus $0 \in W(A)$) with eigenvalues
\[ \lambda_{-m} \le \cdots \le \lambda_{-1} < 0 < \lambda_1 \le \cdots \le \lambda_p,\]
one can say
\[ \theta_j \in (-\infty, \lambda_{-1}] \cup [\lambda_1,\infty), \rlap{\qquad $j=1,\ldots, k$.} \]
Thus \emph{no} harmonic Ritz values call fall in the interval $(\lambda_{-1},\lambda_1)$ containing the origin.
Figure~\ref{fig:hritz_plots} gives a simple illustration contrasting the positive definite and indefinite cases.
For a more detailed discussion of the Hermitian case, see~\cite{Bea98}.

The following result bounds the magnitudes of  harmonic Ritz values via geometric means of the trailing singular values of $A$.\ \ It follows by applying an upper bound for Ritz values~\cite[Theorem~2.3]{CE12} to bound $1/\theta_j$ as a Ritz value of $A^{-1}$.
\begin{theorem}
Suppose $A\in\C^{n\times n}$ and let $\{\theta_j\}_{j=1}^k$ denote the roots of the degree-$k$ GMRES residual polynomial, ordered by increasing magnitude: $|\theta_1| \le |\theta_2| \le \cdots \le |\theta_k|$ (allowing for the possibility that some $\theta_j = \infty$).  Then
\[ |\theta_j| \ge \big(s_n \cdots s_{n-j+1}\big)^{1/j}, \rlap{\qquad $j = 1,\ldots, k,$}\]
where $s_1 \ge s_2 \ge \cdots \ge s_n$ denote the singular values of $A$.
\end{theorem}
 The $k=1$ case of this theorem ensures that no harmonic Ritz values can fall strictly inside the disk of radius $s_n = 1/\|A^{-1}\|$ centered at the origin.  For larger $k$ the bounds quantify the notion that \emph{there cannot be more small magnitude harmonic Ritz values than the number of small singular values of $A$}.\ \ The development of finer containment regions for harmonic Ritz values is an interesting but difficult problem.  When coupled with Proposition~\ref{prop:approx_inv}, such results could illuminate how well  $p(A)$  approximates $A^{-1}$.

\begin{figure}[t!]
\begin{center}
\includegraphics[width=2.15in]{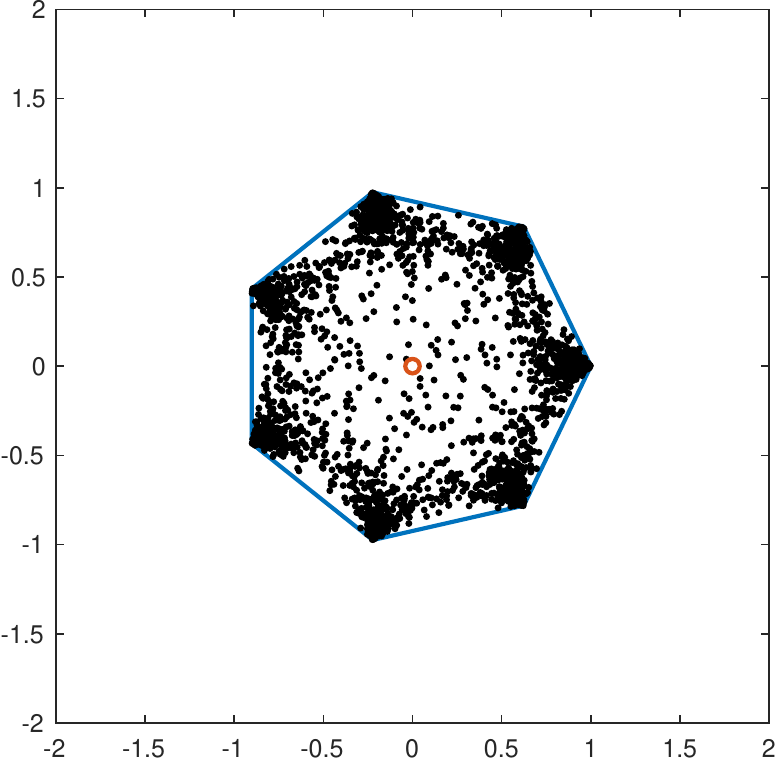}\qquad
\includegraphics[width=2.15in]{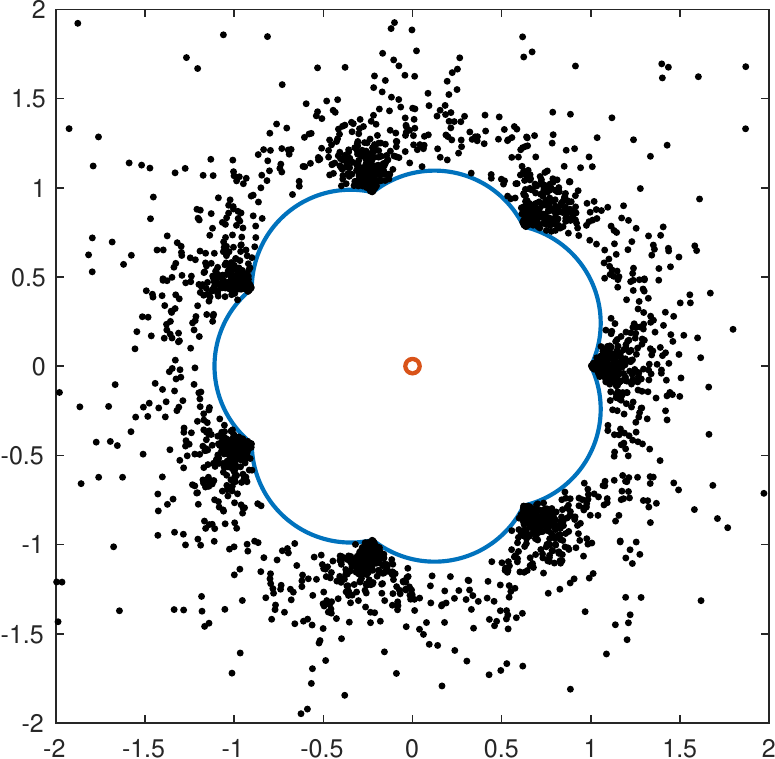}
\end{center}
\vspace*{-5pt}
\caption{\label{fig:rvhrvex}
Ritz values (left) and harmonic Ritz values (right) 
from 5-dimensional random complex Krylov subspaces of the circulant shift ($n=7$),  500~trials.  (83 of the 2500 harmonic Ritz values fall beyond the axes.)  The origin is the red circle; the blue lines show the boundaries of $W(A)$ (left -- an exterior bound on Ritz values) and $1/W(A^{-1})$ (right -- an interior bound on harmonic Ritz values).  }
\end{figure}

Figure~\ref{fig:rvhrvex} contrasts standard Ritz values (eigenvalues of $H_{k,k}$) with harmonic Ritz values (eigenvalues of $H_{k,k} + |h_{k+1,k}|^2 f_k^{} e_k^*$) when $A$ is the circulant shift matrix of order $n=7$ (ones on the superdiagonal and bottom-left entry; zeros elsewhere), a unitary matrix that is notoriously difficult for GMRES~\cite{Bro91}.  The Ritz values all fall in $W(A)$, in this case the convex hull of the spectrum, located 
\emph{inside the unit disk}.  In contrast, the harmonic Ritz values fall in $1/W(A^{-1})$, a nonconvex set \emph{exterior to the unit disk} (except at the eigenvalues).

\subsection{Further tests}

\

{\it Example 2.} 
We explore how the accuracy of the polynomial approximation of $A^{-1}$ obtained from GMRES depends on $b$.
Returning to the convection--diffusion example $- u_{xx} - u_{yy} + \alpha u_{x} + \beta u_y = f $,   we revisit the experiments shown in Figure~\ref{fig:polyacc1}
($\alpha=2$, $\beta=0$: mild departure from normality) and the top-left part of Figure~\ref{fig:polyacc2} ($\alpha=25$, $\beta=10$: significant departure from normality).  
In each case we run 100~experiments with random $b$ vectors (real, normally distributed entries, scaled to be unit vectors) and compare GMRES performance to the accuracy of the corresponding polynomial approximation of the inverse at every iteration (obtained from the GMRES polynomial).  Figure~\ref{fig:loghist} shows the results as colors whose intensities reveal the frequency values out of the 100~experiments: shades of blue show the very consistent behavior of the GMRES residual norm; shades of red show $\|p(A)-A^{-1}\|/\|A^{-1}\|$, which tracks the GMRES residual norm closely but shows a bit more variation.  The color bars on the right of the figure reveal the frequency; summing up values for any iteration number $k$ (i.e., vertically in each plot) will yield~100, the total number of experiments.  The histograms barely overlap (with the red histogram drawn on top of the blue), as can be seen by the solid blue and red lines that show the extreme behavior of the 100~experiments.

\begin{figure}
\begin{center}
   \includegraphics[height=1.8in]{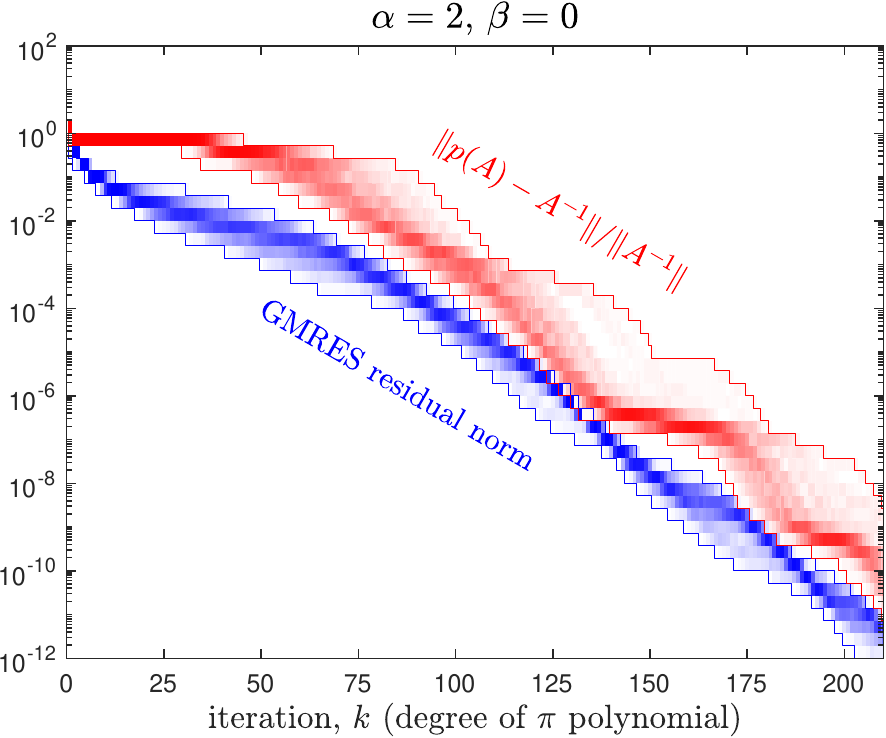}\quad
   \includegraphics[height=1.8in]{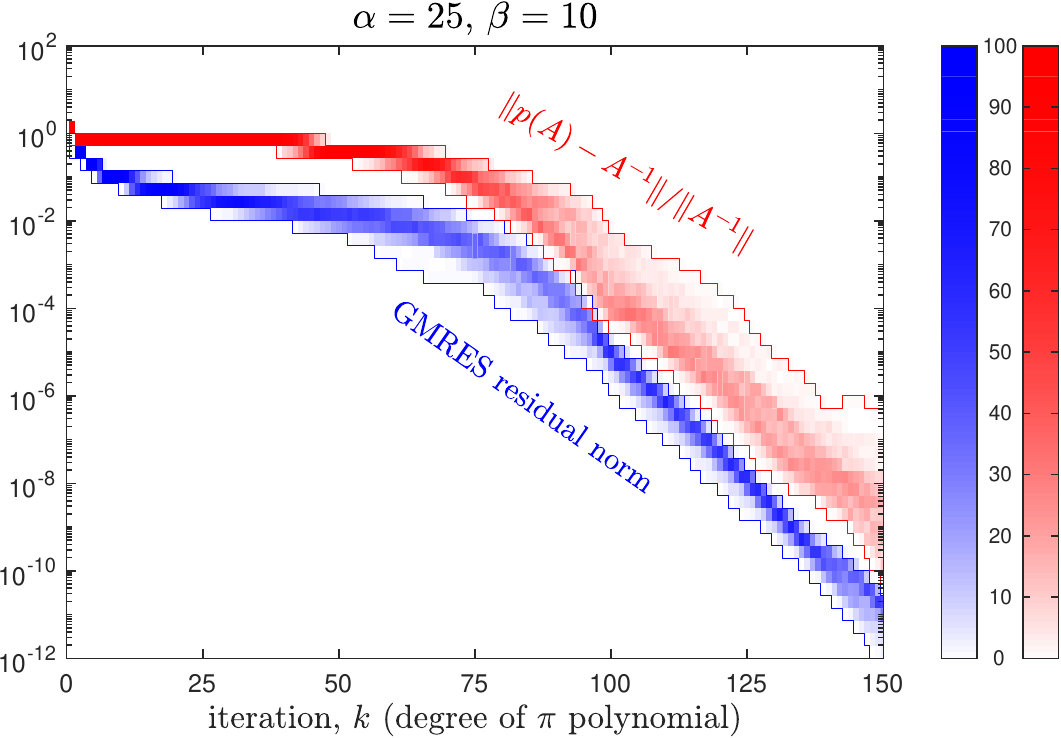}
\end{center}

\vspace*{-5pt}
\caption{Repetitions of Figure~\ref{fig:polyacc1} (left) and the top-left part of Figure~\ref{fig:polyacc2} (right) based on discretizations of the convection--diffusion equation $- u_{xx} - u_{yy} + \alpha u_{x} + \beta u_y = f $ of size $n=2500$.  The colors show the frequency with which the GMRES residual norm (blue) and polynomial inverse approximation (red) reached a given tolerance at iteration $k$, out of 100 experiments with random $b$ vectors.  (The dark blue and red lines show the envelopes of these barely overlapping histograms.)}
\label{fig:loghist}
\end{figure}


{\it Example 3.}  We examine the effect of ill-conditioning on the accuracy of the polynomial approximation of $A^{-1}$.\ \  To focus on this effect, we take $A$ to be a positive definite diagonal matrix with no outstanding eigenvalues.  (No extra stabilization roots are required in this experiment.)  For dimension $n=2501$ we set
\[ \widehat{A} = {\rm diag}(1^p, 2^p, \ldots, 1251^p, 2\cdot (1251^p)-1250^p, 2\cdot (1251^p)-1249^p, \ldots, 2\cdot(1251^p) - 1^p)\]
and scale $A = \widehat{A}/(1250.5)^{p-1}$.
As the parameter $p\ge 0$ increases, the conditioning grows and the eigenvalues evolve from clustered in the middle of the spectrum ($p<1$) to clustered at both ends ($p>1$).
We take 20~values of $p\in [0.55,1.8]$, in each case running GMRES with the same $b$ (a random normal vector, scaled to unit norm).

This example shows how the degree of the polynomial varies with the condition number, and that a high degree polynomial can be effective.  
Figure~\ref{fig:polydegr} shows the degree needed for full GMRES to converge to residual norm below $10^{-10}$.  
The relative accuracy of $p(A)$ ranges from a maximum of $4.4\times 10^{-9}$ for $p\approx 0.616$ to a minimum of $3.6\times 10^{-11}$ for $p=1.8$.
Even with ill-conditioned problems, $p(A)$ can accurately approximate $A^{-1}$, though the degree of $p$ grows with the ill-conditioning.

\begin{figure}
\begin{center}
\includegraphics[width=3.75in]{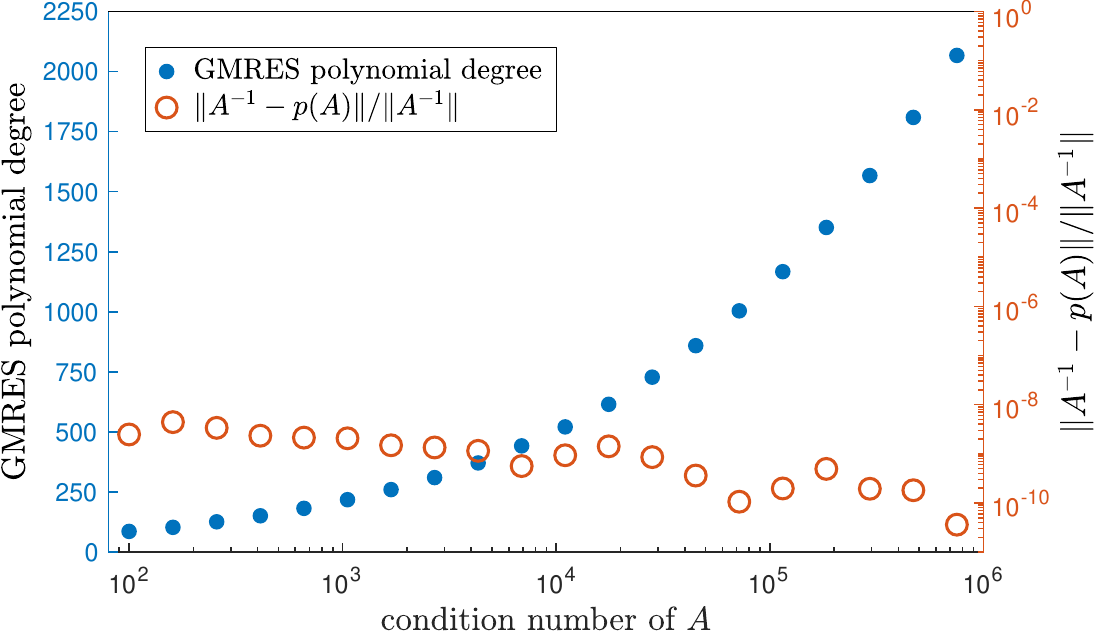}
\end{center}

\vspace*{-5pt}
\caption{Example~3:  diagonal matrices $(n=2501)$ with no outstanding eigenvalues.  The degree of the GMRES polynomial when the residual goes below $10^{-10}$ and the corresponding relative error in the polynomial approximation of $A^{-1}$ is plotted versus the condition number.}
\label{fig:polydegr}
\end{figure}

\section{Alternative methods for building the polynomial} \label{sec:alt}

Finding the polynomial $p$ that approximates $A^{-1}$ can be expensive due to the need to run full (non-restarted) GMRES, possibly for many iterations.  The orthogonalization cost and the storage increase as the iteration proceeds.  We discuss three ways of reducing this cost: generating the polynomial $p$ from restarted GMRES, a double polynomial formed by polynomial preconditioned GMRES, and the nonsymmetric Lanczos process.  

\subsection{Restarted GMRES}

The restarted GMRES algorithm~\cite{SaSc} can significantly reduce the orthogonalization cost, restricting the GMRES optimization to Krylov subspaces of fixed dimension $m\ll n$.  However, restarted GMRES often takes more iterations than full GMRES, so a higher degree polynomial is needed.

When GMRES($m$) is run for $c$ cycles, the overall residual polynomial is a product of the residual polynomials for each cycle, $\pi(z) = \pi_1(z)\cdots \pi_c(z)$, where each $\pi_j$ is a degree-$m$ polynomial. 
The roots of $\pi(z)$ are all the roots of the residual polynomials (harmonic Ritz values) from each GMRES cycle put together. 
The polynomial $p$ is defined via $\pi(z) = 1 - zp(z)$.  Then $p(A)$ can be multiplied against a vector using~\cite[alg.~3]{PPGStable} with the list of roots from the overall $\pi$. 
When each restart occurs at the end of a cycle of $m$ iterations, $m$ of these roots become ``locked in,'' and will henceforth be roots of the $\pi$ that is built up over future cycles.  (In contrast, for full GMRES \emph{all} roots of $\pi$ typically change at each iteration.)  The harmonic Ritz values may nearly recur in cyclic patterns across restarted GMRES cycles~\cite{BaJeMa,CC}, which could limit the effectiveness of this approach to designing polynomial approximations to the inverse.  (In the ``hybrid GMRES'' approach~\cite{NaReTr},  $m$ iterations of GMRES are run, and the polynomial is constructed from $\pi(z) = (1-z/\theta_1)^c \cdots (1-z/\theta_m)^c$.  The resulting approximation $p(z)$ will interpolate $1/z$ and its first $c-1$ derivatives at $\theta_1, \ldots, \theta_m$.)

\subsection{Double polynomials}
An alternative approach to building high-degree polynomials while controlling orthogonalization costs uses the composition of two polynomials, as generated by polynomial preconditioned GMRES.\ \ 
We call such an approximation to the inverse a \emph{double polynomial}~\cite{PPArn,PPGStable}. The double polynomial often needs to be of higher degree than a single GMRES polynomial to deliver the same accuracy, but it can be of lower degree than the restarted GMRES polynomial.  In general, the double polynomial makes approximation of $A^{-1}$  much more practical.  

We describe how to build the double polynomial from polynomial preconditioned GMRES (Subsection~\ref{ssec:ppgmres}).  First, select the degree of the $\phi$ polynomial in polynomial preconditioned GMRES.\ \  We regard this as an inner polynomial, so we call it $\phi_{\rm in}$ and its degree $\dphiin$.  Run GMRES on $A$ for $\dphiin$ iterations, and compute the roots of the GMRES residual polynomial $\pi_{\rm in}$.  These roots define the polynomials $\phi_{\rm in}$ and $p_{\rm in}$ according to $\pi_{\rm in}(z) = 1 - \phi_{\rm in}(z) = 1 - z@p_{\rm in}(z)$.  Next, run GMRES a second time (``outer GMRES"), now on the polynomial preconditioned system $Ap_{\rm in}(A)y = b$ (or $\phi_{\rm in}(A)y = b$), solving to the required tolerance.  (The approximate solution of the linear system is $x = p_{\rm in}(A) y$.)   Compute the roots of the residual polynomial from this outer GMRES run; they define the polynomials $p_{\rm out}$ and $\phi_{\rm out}$.  The degree of $\phi_{\rm out}$, say $d\phi_{\rm out}$, equals the number of outer GMRES iterations.  The overall polynomial $p$ that approximates the inverse is $p(z) = p_{\rm in}(z) p_{\rm out}(\phi_{\rm in}(z))$, of degree $d\phi_{\rm in} \times d\phi_{\rm out} - 1$.

\subsection{Nonsymmetric Lanczos}
The nonsymmetric Lanczos algorithm can solve linear equations in various ways, the most straightforward of which is the BiConjugate Gradient (BiCG) method \cite{Sa03}.  The BiCG residual polynomial $\pi$ has roots at the eigenvalues of the tridiagonal matrix built by the Lanczos process, which determine the inverse approximating polynomial $p$ through $\pi(z) = 1 - zp(z)$.  (Future work should explore more sophisticated approaches based on nonsymmetric Lanczos, building on Thornquist's investigation of BiCGSTAB and TFQMR~\cite{Tho06}.)

\begin{figure}[t!]
\begin{center}
\includegraphics[height=2.0in]{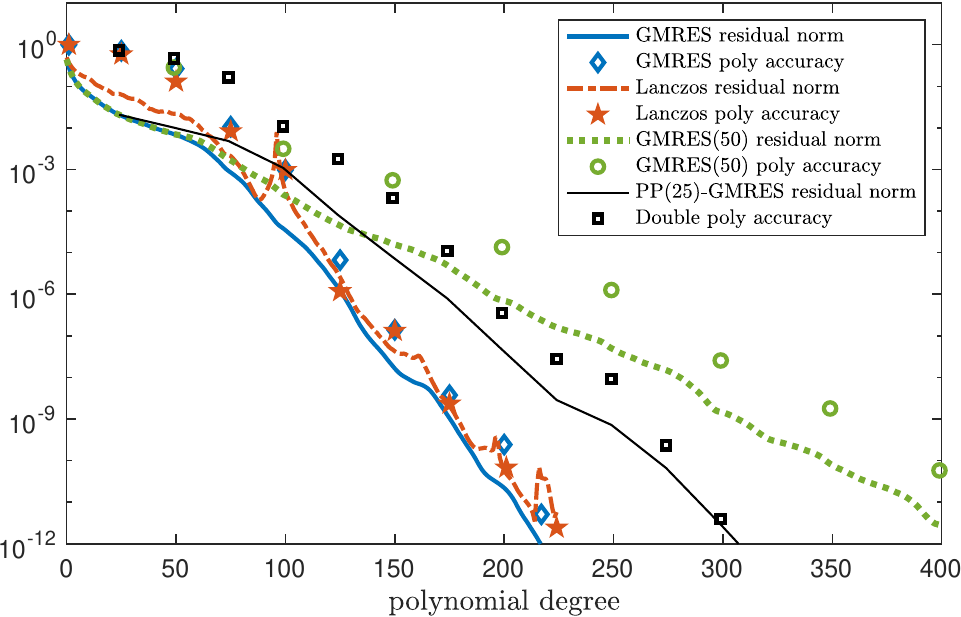}
\end{center}

\vspace*{-6pt}
\caption{Comparison of four methods for building polynomial approximations to $A^{-1}$, for the convection--diffusion problem from Example~1 ($n=2500$).  The lines show the residual norms of the linear system; the markers show the relative accuracy of the polynomial approximation to $A^{-1}$.}
\label{fig:AltMeths}
\end{figure}

\subsection{Comparison}
\ 
{\it Example 4.} 
We return to the matrix from Example~1 of dimension~2500 corresponding to  $- u_{xx} - u_{yy} + 2 u_{x} = f $.
Figure~\ref{fig:AltMeths} shows results for polynomials found by  GMRES, restarted GMRES, the double polynomial from polynomial preconditioned GMRES, and nonsymmetric Lanczos.  The linear equation residual norms are shown, along with the relative accuracy $\|A^{-1} - p(A)\| / \|A^{-1}\|$.

Nonsymmetric Lanczos has low orthogonalization costs and finds polynomials that are almost as low degree for a given accuracy as those from full GMRES.\ \  However, it requires more matrix-vector products, multiplying with both $A$ and $A^*$ at each iteration.  The method can also be unstable, and for an indefinite matrix, Ritz values can  fall near the origin. 
The accuracy of the double polynomial generally stays ahead of the restarted GMRES polynomial, and requires even less orthogonalization. 
Thus we will mostly use the double polynomial in the rest of this paper.

\section{Linear Equations with Multiple Right-hand Sides} \label{sec:mult_rhs}

Given a good polynomial approximation to $A^{-1}$, one can quickly solve additional linear systems: 
simply multiply each right-hand side by $p(A)$, as in Algorithm~\ref{Alg:MultRHS}. 
One simply needs to represent $p$ via roots of $\pi$, whether obtained from full GMRES, or one of the alternatives discussed in the last section.  
The process for double polynomial preconditioning is a bit more subtle, and so is detailed in Algorithm~\ref{Alg:MultRHSdoub}.

\begin{algorithm}[h!]
    \caption{Solve multiple right-hand sides with a polynomial from GMRES}
    \begin{description}
    \item[1.] {\bf First right-hand side system.} Solve the first system $Ax^{(1)}=b^{(1)}$ to the requested relative residual tolerance with full GMRES.\ \ Find the harmonic Ritz values from the last GMRES iteration, and use them to build a polynomial approximation $p(A)$ to $A^{-1}$.\ \   Add roots for stability, if necessary (\cite[alg.~2]{PPGStable}).
    \item[2.] {\bf Additional right-hand side systems.} Apply $p(A)$ to the right-hand sides of the other systems.  For the systems $Ax^{(j)}=b^{(j)}$, $j=2:\nrhs$, compute $p(A)b^{(j)} \approx x^{(j)}$ using \cite[alg.~3]{PPGStable}.
\end{description}
\label{Alg:MultRHS}
\end{algorithm}

\begin{algorithm}[t!]
    \caption{Solve multiple right-hand sides with a double polynomial}
    \begin{description}
    \item[1.] {\bf First right-hand side.} Choose the degree $\dphiin$ of the inner polynomial $\phi_{\rm in}$.  Run $\dphiin$ iterations of GMRES on the first right-hand side system $Ax^{(1)}=b^{(1)}$ to find the roots of the residual polynomial $\pi_{\rm in}$.  These roots define the inner polynomials $p_{\rm in}$ and $\phi_{\rm in}$.  Solve the first right-hand side system to the requested relative residual tolerance with polynomial preconditioned GMRES, PP($\dphiin$)-GMRES.\ \  The harmonic Ritz values from this run of GMRES give the roots of the outer GMRES residual polynomial; they define the outer polynomial $p_{\rm out}$.
    \item[2.] {\bf Additional right-hand side systems.} For the systems $Ax^{(j)}=b^{(j)}$, $j=2:\nrhs$, compute $p_{\rm in}(A)p_{\rm out}(\phi_{\rm in}(A)) b^{(j)} \approx x^{(j)}$ using \cite[alg.~1]{PPGStable} for $\phi_{\rm in}(A)$ and~\cite[alg.~3]{PPGStable} for both $p_{\rm in}$ and $p_{\rm out}$.
\end{description}
\label{Alg:MultRHSdoub}
\end{algorithm}

{\it Example 5.}   Consider the discretization of  $- u_{xx} - u_{yy} + 2 u_{x} - 10^2 u = f $ on the unit square with 200 interior grid points in each direction, giving $A$ of dimension $n=40{,}000$.  The $-10^2 u$ term makes the matrix indefinite and poses a challenge for conventional solvers like restarted GMRES and BiCGStab~\cite{vdV92}. We seek to solve 10~systems with random  right-hand sides (normally distributed entries, scaled to unit norm) to  residual norm $10^{-8}$.  The first two rows of Table~\ref{Tab:CDg10} show the results of standard methods.
Restarted GMRES(100) with a limit of 40{,}000 cycles not only takes over 28~hours; it only solves five systems to accuracy below $10^{-8}$ (others as high as $6\times 10^{-3}$).  BiCGStab with a limit of 100{,}000 iterations takes 380~seconds and only solves four systems to $10^{-8}$ (others as high as $3.5\times 10^{-4}$).  
The remaining rows show results for polynomial inverse approximation.
First, running full GMRES to relative residual tolerance of $10^{-11}$ creates a degree~1345 polynomial $p$ (which includes two extra roots added for stability control; see Section~7).  The next nine systems $Ax^{(j)} = b^{(j)}$ are solved by multiplying $p(A) b^{(j)}$, and all reach residual norms below $6 \times 10^{-9}$.   Finding $p$ and solving the first system takes 162~seconds,  while the next nine systems take only 2.4~seconds. 
The double polynomial generated with PP(40)-GMRES gives a polynomial of higher degree, 2039. (The outer $\phi$ polynomial is originally degree 50 and one stability root is added, giving the overall degree $40\times 51-1 = 2039$.)  This method takes only 1.1~seconds to solve the first system and generate the approximation to $A^{-1}$; the following nine systems are solved in an additional 3.1 seconds, all reaching residual norm of $7.5\times 10^{-11}$ or better.  
Here the accuracy in solving the multiple right-hand side systems is about an order of magnitude behind the GMRES residual norm in developing the polynomial.  However, it is not always possible to know ahead of time what GMRES residual norm is needed (see above that the single polynomial from GMRES was not as accurate). 
For the last line of the table we use nonsymmetric Lanczos to construct $p$.  The method is also very quick (more time for finding the polynomial, but even faster for solving the other systems, to similar accuracy).  

\begin{table}
\caption{Example~5:  convection-diffusion matrix from $- u_{xx} - u_{yy} + 2 u_{x} - 10^2 u = f $ of size $n = \mbox{40,000}$ with 10 random right-hand sides.  Compare using various $p(A) \approx A^{-1}$ to standard methods.}

\vspace*{-3pt}
\begin{center}
\begin{tabular}{|c|c|c|} \hline\hline
method      & total MVP's  & time \\ \hline \hline
\rule[-6pt]{0pt}{16pt}GMRES(100)   & 21.5 million  & 28.6 hours     \\ \hline
\rule[-6pt]{0pt}{16pt}BiCGStab    & 1.37 million   & 380 seconds     \\ \hline\hline
\rule[-6pt]{0pt}{16pt} $p(A)$, deg = 1345    & 13{,}451 & 162 + 2.4 seconds  \\ \hline
Double Polynomial & \smash{\raisebox{-6pt}{20{,}749}}       & \smash{\raisebox{-6pt}{1.1 + 3.1 seconds}}\\
${\rm deg} = 40\times 51-1 = 2079$        &   &       \\ \hline
Nonsymmetric Lanczos & \smash{\raisebox{-6pt}{14{,}456}}  & \smash{\raisebox{-6pt}{3.4 + 2.3 seconds}}   \\
deg = 1314  &       &       \\ \hline
\hline
\end{tabular}
\end{center}
\label{Tab:CDg10}
\end{table}

{\it Example 6.} 
We use the matrix from the last example but make it larger at $n=160{,}000$ and try different amounts of indefiniteness (different values of $\gamma$).  We also experiment with different levels of standard preconditioning.  We compare solving 10~random right-hand side systems with BiCGStab (limit of 100{,}000 iterations) and with double polynomial approximations to $A^{-1}$.\ \ 
PP-GMRES is run to residual norm of $10^{-10}$ to generate the polynomial.  The degree of the inner polynomial $\phi_{\rm in}$ is~50 in all tests except for the case of $\gamma=20$ with no standard preconditioning.  For that case it is set to degree~100, because the degree~50 case is not sufficiently stable: the outer polynomial adds too many extra roots (17 stability roots added to a degree 120 polynomial).  This case points out that while our new polynomial approximation method can be very effective, additional research into stabilization is still warranted.  (Stability is discussed further in Section~7.)  

Table~\ref{Tab:CDindef} shows the results.  (ILU(0) refers to incomplete factorization with no fill-in, while ILU(.01) is MATLAB's Crout incomplete factorization with drop tolerance~0.01.)
Matrix-vector products are given per right-hand side, but the time is the total for all 10~systems.
The double polynomial degree and runtime increase with indefiniteness.  For $\gamma=20$ without preconditioning, the polynomial has composite degree of $12{,}599$.  However, this polynomial is much more effective than BiCGStab, which takes more time and does not give accurate answers.  With the more powerful preconditioning of ILU(.01), BiCGStab is effective, but still runs slower than the double polynomial approach.  

\begin{table}
\caption{Example~6: convection-diffusion matrix ($n=160{,}000$) from $- u_{xx} - u_{yy} + 2 u_{x} - \gamma^2 u = f $ with 10~random right-hand sides.  Compare  a double polynomial to BiCGStab for different values of $\gamma$ and different levels of preconditioning.  ILU(0) refers to incomplete factorization with no fill-in, while ILU(.01) is MATLAB's Crout incomplete factorization with drop tolerance~0.01.}

\vspace*{-3pt}
\begin{center}
\begin{tabular}{|c|c|c|c|c||c|c|c|}  \hline\hline
& & \multicolumn{3}{|c||}{Double Polynomial}                  &  \multicolumn{3}{|c|}{BiCGStab }              \\  \hline
\smash{\raisebox{-5pt}{$\gamma$}} & precon-  & degree of    & max res.          & time          & average       & max res.          & time      \\ 
                          & ditioner             & polynomial    & norm              & \small{(sec)}     & MVP's         & norm              & \small{(sec)}   \\ \hline \hline
 & none         & $1899$   & $4.0\times10^{-10}$    & 25           & 1666           & $1.0\times10^{-8}$     & 31       \\ \cline{2-8}
0 & ILU(0)       & $549$   & $7.9\times10^{-10}$    & 31           & 496           & $9.8\times10^{-9}$     & 30       \\ \cline{2-8}
 & ILU(.01)     & $349$   & $3.6\times10^{-11}$    & 74           & 245           & $9.9\times10^{-9}$     & 67       \\ \hline \hline
 & none        & $3749$  & $1.7\times10^{-9}$     & 49           & 198{,}106     & $3.0\times10^{-6}$     & 3885       \\ \cline{2-8}
10 & ILU(0)      & $1399$  & $1.2\times10^{-9}$     & 81           & 15{,}628     & $4.9\times10^{-7}$     & 927       \\ \cline{2-8}
 & ILU(.01)     & $1149$  & $4.3\times10^{-10}$     & 127           & 2519     & $9.9\times10^{-9}$     & 228       \\ \hline \hline

 & none       & $12{,}599$  & $1.9\times10^{-9}$    & 171            & 197{,}896      & $9.9\times10^{-3}$     & 4078       \\ \cline{2-8}
20 & ILU(0)     & $4499$     & $2.4\times10^{-8}$    & 236            & 26{,}886      & $9.6\times10^{-9}$     & 1604       \\ \cline{2-8}
 & ILU(.01)   & $2799$  & $2.8\times10^{-10}$    & 235            & 5457     & $9.8\times10^{-9}$     & 420       \\ \hline \hline
        
\end{tabular}
\end{center}
\label{Tab:CDindef}
\end{table}

{\it Example 7.}  
We investigate two matrices from SuiteSparse~\cite{SuiteSparse}: Bcircuit ($n=68{,}902$) and Circuit\_3 ($n=12{,}127)$.  There are 10 random right-hand side systems.  BiCGStab is run with the limit of 100{,}000 iterations.    PP-GMRES is run to $10^{-10}$ to generate a double polynomial approximation to $A^{-1}$. For both tests, the degree of the inner polynomial $\phi_{\rm in}$ is 100 for ILU(0) and 50 for ILU(.01).  For this example, we change the $\pofcutoff$ (see Subsection 1.2) to $4$, which gives slightly better results.  These results are in Table~\ref{Tab:SuiteSp}.

For the matrix Bcircuit with ILU(0) preconditioning, the new polynomial method gives much better results.  All 10 right-hand sides are accurate to $1.1 \times 10^{-8}$ or better.  BiCGStab gives poor results for most right-hand sides (only two have residual norm below $10^{-5})$ and it takes more time.  BiCGStab is sensitive to this difficult problem.  However, the polynomial method can also be sensitive.  With the degree of the inner polynomial $\phi_{\rm in}$ set to 50, the spectrum of $\phi_{\rm in}(A)$ has one outstanding eigenvalue, and while the outer polynomial adds 11 roots for stability (so degree 129+11), it is still not stable.  The composite polynomial gives residual norms over $10^{10}$ for right-hand sides 2 through 10.  Similar to one case of the previous example, the polynomial method is sensitive to the choice of inner polynomial degree.  Instability sometimes happens with these preconditioned SuiteSparse matrices, because the preconditioning causes one or a few very outstanding eigenvalues.  Next, with ILU(.01) preconditioning, the problem is easier and both methods are effective.  The polynomial approach gives more accuracy and runs somewhat faster.  

Next we consider the matrix Circuit\_3.  For ILU(0), the polynomial method struggles, giving residual norms as high as $3.4 \times 10^{-3}$.  However, it is much faster and effective than BiCGStab, which makes no progress.  For the better preconditioning, the polynomial method works well, but BiCGStab still struggles (six of the systems end up with residual norms above $10^{-2}$).  If even more powerful preconditioning is applied, it is possible to make BiCGStab effective and have the two approaches give similar results.

\begin{table}
\caption{Example~7:   Compare BiCGStab to using a double polynomial approximation to the inverse  for solving 10 random right-hand side systems.  Two matrices are used (Bcircuit, $n=68{,}902$ and 
Circuit\_3, $n=12{,}127$) with two different levels of preconditioning.  ILU(0) refers to incomplete factorization with no fill-in, while ILU(.01) is MATLAB's Crout incomplete factorization with drop tolerance 0.01. }

\vspace*{-3pt}
\begin{center}
\begin{tabular}{|c|c|c|c|c||c|c|c|}  \hline\hline
& & \multicolumn{3}{|c||}{Double Polynomial}   &  \multicolumn{3}{|c|}{BiCGStab }   \\  \hline
\smash{\raisebox{-7pt}{matrix}} & precond-   & degree     & max res.          & time          & average       & max res.          & time      \\ 
        &    itioner    & of poly    & norm              & \small{(sec)}     & MVP's         & norm              & \small{(sec)}   \\ \hline \hline
\smash{\raisebox{-7pt}{Bcircuit}}   &  ILU(0)   & $6766$   & $1.1\times10^{-8}$     & 134      & 9179           & $1.1\times10^{-2}$     & 210       \\ \cline{2-8}
            &  \!ILU(.01)\! & $1611$  & $1.4\times10^{-10}$     & 26       & 1703     & $1.0\times10^{-8}$     & 36       \\ \hline
       \hline

\smash{\raisebox{-7pt}{\!Circuit\_3\!}}  &  ILU(0)   & $8045$   & $3.4\times10^{-3}$     & 16    & 28{,}947  & $1.0$     & 128       \\ \cline{2-8}
            & \!ILU(.01)\!  & $1517$  & $6.3\times10^{-10}$     & 3.7    & 1081     & $9.2\times10^{-1}$     & 4.3       \\ \hline
     \hline

\end{tabular}
\end{center}
\label{Tab:SuiteSp}
\end{table}

We finish this section with a theorem about accuracy of multiple right-hand side systems for diagonalizable $A$.\ \ If one finds $p(A)$ from GMRES applied to $b^{(1)}$ and then uses it to solve a system with right-hand side $b^{(2)}$, the size of the residual $r^{(2)} = b^{(2)}-Ap(A)b^{(2)}$ can be bounded by a term that depends on the potential deficiency of $b^{(1)}$ in any eigenvectors of $A$ (times a term that describes the departure of $A$ from normality).

\begin{theorem}
    Let $A\in\C^{n\times n}$ be diagonalizable, $A = Z \Lambda Z^{-1}$, and let $Ax^{(1)}=b^{(1)}$ and $Ax^{(2)}=b^{(2)}$ be two  systems with unit norm right-hand sides.  Let $A$ have eigenvalues $\lambda_1, \ldots, \lambda_n$ 
    and corresponding eigenvectors $z_1, \ldots, z_n\in\C^n$.  Expand the right-hand sides as $b^{(1)} = Z(Z^{-1}b^{(1)}) = \sum\beta_i^{(1)} z_i$ and $b^{(2)} = Z(Z^{-1}b^{(2)}) = \sum\beta_i^{(2)} z_i$.  Let the first residual vector be $r^{(1)} = b^{(1)} - A \wh{x}^{(1)}$, where $\wh{x}^{(1)}$ is the GMRES approximate solution with $\wh{x}^{(1)} = p(A)b^{(1)}$. 
    Obtain the second residual by premultiplying this $p(A)$ generated from the first linear system against the second right-hand side: $r^{(2)} = b^{(2)}-A \wh{x}^{(2)}$, where $\wh{x}^{(2)}  = p(A)b^{(2)}$. Then 
    \[    \|r^{(2)}\| \leq  \left(\|Z\| \|Z^{-1}\| \max\bigg|\frac{\beta_i ^{(2)}}{\beta_i ^{(1)}}\bigg|\right) \|r^{(1)}\|.\]
\end{theorem}

\begin{proof}
Using that $r^{(2)} = b^{(2)}-A \widehat{x}^{(2)} = b^{(2)}-A p(A)b^{(2)} = \pi(A)b^{(2)}$ and the eigenvalue decomposition of $\pi(A)$,
\begin{equation} \label{eq:r1}
r^{(2)} = \pi(A)b^{(2)} 
= Z\pi(\Lambda)Z^{-1} b^{(2)} 
= \sum_{i=1}^n \beta_i ^{(2)} \pi(\lambda_i)z_i.
\end{equation}
Note that $\|r^{(2)}\| \le \|Z\| \|Z^{-1} r^{(2)}\|$, and since $\Lambda$ is diagonal, 
\begin{align*}
\|Z^{-1}r^{(2)}\|^2 \ &=\  
\|\pi(\Lambda) (Z^{-1}b^{(2)})\| 
 = \sum_{i=1}^n \big| \beta_i^{(2)} \pi(\lambda_i)  \big|^2 
= \sum_{i=1}^n \bigg| \frac{\beta_i^{(2)}}{\beta_i^{(1)}}\cdot \beta_i^{(1)} \pi(\lambda_i)\bigg|^2 \\
&\leq \max\bigg|\frac{\beta_i ^{(2)}}{\beta_i ^{(1)}}\bigg|^2\ \sum_{i=1}^n \big|\beta_i^{(1)} \pi(\lambda_i) \big|^2 \\ 
\ &=\ \max\bigg|\frac{\beta_i ^{(2)}}{\beta_i ^{(1)}}\bigg|^2 \|\pi(\Lambda) Z^{-1} b^{(1)}\|^2
\ =\ \max\bigg|\frac{\beta_i ^{(2)}}{\beta_i ^{(1)}}\bigg|^2 \|Z^{-1} r^{(1)}\|^2.
\end{align*} 
Thus $\|r^{(2)}\| \le 
\|Z\| \|Z^{-1}\| \big(\max_i |\beta_i^{(2)}/\beta_i^{(1)}|\big) \|r^{(1)}\|.$
\end{proof}

\section{Deflated polynomials for solving multiple right-hand sides} \label{sec:deflated}

Small magnitude eigenvalues slow the convergence of iterative methods.  After solving the first linear system with GMRES, one can use the resulting Krylov subspace to obtain approximations to the eigenvectors associated with the problematic small eigenvalues.  These approximate eigenvectors can then be projected out from subsequent right-hand sides, greatly reducing the influence of the small eigenvalues and thus expediting convergence.  
This process, briefly mentioned in Subsection~\ref{sec:reviewsub3}, is called \emph{deflation}.
Here we explain how to integrate this idea with polynomial approximations of $A^{-1}$ for lowering the required polynomial degree, and then solving subsequent linear systems.

Algorithm~\ref{Alg:MultRHSDeflated} describes three main steps for solving $Ax^{(j)} = b^{(j)}$ for $j=1,\ldots, \nrhs$.  
\emph{Step~1.}~Solve the first system using polynomial preconditioned GMRES.\ \ From the resulting outer subspace, find approximate eigenvectors via a standard Rayleigh--Ritz procedure.  Use these  to deflate the most significant small eigenvalues from the other right-hand systems.  \emph{Step~2.} Solve the deflated second system with polynomial preconditioned GMRES using the same inner polynomial as before.  This solve develops a double polynomial, which we call a \emph{deflated polynomial}.  If the deflation is effective for this second system, the polynomial degree is reduced. \emph{Step~3.}~Apply this deflated polynomial to solve the systems with all other (deflated) right-hand sides.  

\begin{algorithm}
    \caption{A deflated double polynomial for solving multiple right-hand sides}
    \begin{description}
    \setcounter{enumi}{-1}
    \item[0.] {\bf Preliminary.}  Choose $\dphiin$, the degree of the inner polynomial $\phi_{\rm in}$.  Pick $\nev$, the number of eigenvalues to deflate.  Choose relative residual norm tolerances for the three steps: the first system is solved to $\rtol_1$, the second to $\rtol_2$ and $\rtol_3$ for the other systems.
    \item[1.] {\bf Solve first system and compute approximate eigenvectors.} Run GMRES for $\dphiin$ iterations to develop the inner polynomials $p_{\rm in}$ and $\phi_{\rm in}$.  Then solve the first right-hand side system with polynomial preconditioned GMRES, i.e., run PP($\dphiin$)-GMRES with $\rtol_1$.  From the subspace thus developed, apply the Rayleigh--Ritz procedure to compute approximate eigenvectors corresponding to the $\nev$ smallest eigenvalues~\cite{PPArn}.
    \item[2.] {\bf Develop the deflated polynomial with the second system.} Deflate the second right-hand side system with a Galerkin projection (alg.~2 in \cite{MGLE}) over the approximate eigenvectors (solve $V^*\kern-2pt AV d^{(2)} = V^* b^{(2)}$, then $x^{(2)}_e = V d^{(2)}$ and $r^{(2)} = b^{(2)} - A x^{(2)}_e$, with the deflated system being $A(x^{(2)} -x^{(2)}_e) = r^{(2)}$).  Then apply PP($\dphiin$)-GMRES, with the same inner polynomial developed earlier, to the deflated second system with $\rtol_2$.  This generates an outer polynomial $p_{\rm out}$ and thus the deflated double polynomial $p(z) = p_{\rm in}(z) p_{\rm out}(\phi_{\rm in}(z))$.  
    \item[3.] {\bf Other systems.}  For $A x^{(j)} = b^{(j)}$, $j=3,\ldots,\nrhs$, project over the approximate eigenvectors to get the partial solution $x_e^{(j)}$ and deflated system $A (x^{j} - x_e^{(j)}) = r^{(j)}$.  The approximate solution to the original system is $x^{(j)} = x_e^{(j)} + p(A) r^{(j)}$.   
    \item[4.] {\bf Optional.}  Reapply the deflated polynomial to systems that have not converged to $\rtol_3$, including the second system if more accuracy than $\rtol_2$ is needed.  Each  reapplication needs a projection then multiplication by $p(A)$.
\end{description}
\label{Alg:MultRHSDeflated}
\end{algorithm}

For a significantly nonnormal matrix, deflation is less effective unless both right and left eigenvectors are computed.  
For these cases, we will consider applying a lower degree polynomial more than once; see Example~9. 

{\it Example 8.}  We use the matrix BWM2000 from SuiteSparse~\cite{SuiteSparse}.  
Though not large ($n=2000$) and only mildly indefinite (all eigenvalues have negative real parts, except for two that are barely positive), solving linear equations with this matrix can be difficult.  The right-hand sides are again random unit vectors.  The parameters are $d\phi_{\rm in}=50$, $\nev=30$, $\rtol_1 = 10^{-11}$, $\rtol_2 = 10^{-9}$, and $\rtol_3 = 10^{-8}$.  PP(50)-GMRES takes 64~iterations and finds 30 approximate eigenvectors with residual norms ranging from $6.5\times 10^{-8}$ to $4.7\times 10^1$.  Deflated PP-GMRES on the second system takes 16 iterations.  Figure~\ref{Fig:pAif20} shows the convergence of PP(50)-GMRES on the \emph{first} system with a (blue) dashed line and BiCGStab on the same system with a (yellow) dash-dot line.  
BiCGStab converges only to residual norm $4.6\times10^{-3}$ in 45{,}405 matrix-vector products (most not shown on the plot).  
The (red) solid line shows the solution of the deflated system corresponding to the \emph{second} right-hand side, which proceeds much quicker due to the deflation.

\begin{figure}[t!]
\begin{center}
\includegraphics[height=2.0in]{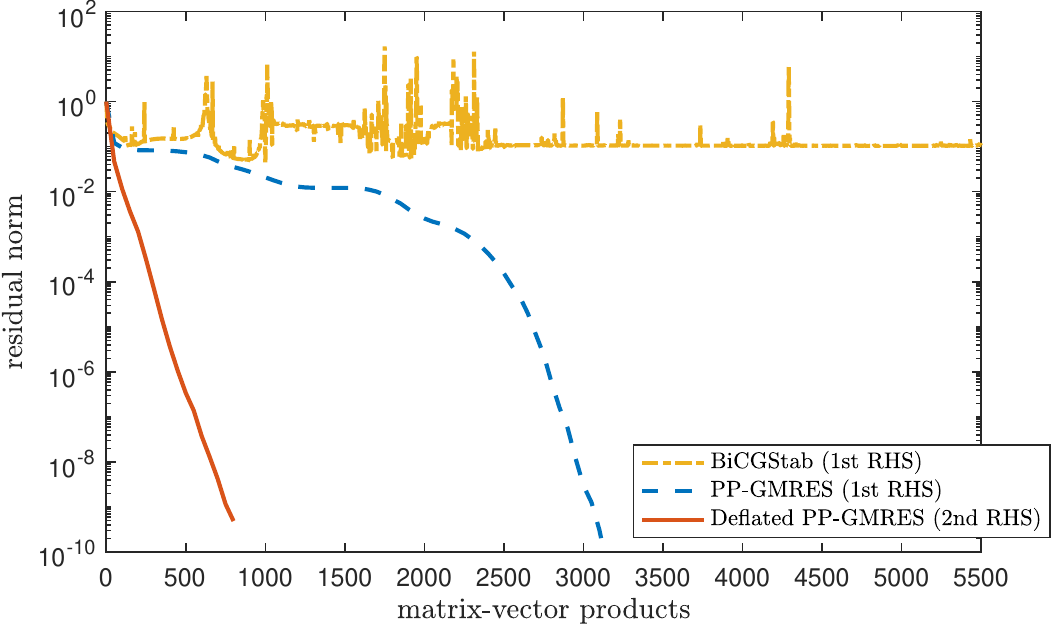}
\end{center}
\vspace{-7pt}
\caption{Example 8: the matrix is BWM2000.  Convergence is shown for BiCGStab, polynomial preconditioned PP(50)-GMRES with no restarting, and for deflated PP(50)-GMRES.}
\label{Fig:pAif20}
\end{figure}

Table~\ref{Tab:BWM} shows the cost for solving 10~right-hand sides.  The deflated double polynomial method uses 0.37 seconds to solve the first system and compute approximate eigenvectors.  Then 0.09 seconds are used for the second system and generating a deflated polynomial.  The other systems take 0.07 seconds, and all solutions have residual norms of $8.3\times10^{-10}$ or smaller.  BiCGStab is applied to the ten systems with tolerance of $10^{-8}$ and a maximum of 100,000 iterations.  This process takes 6.9 seconds, but only three of the systems converge.  The log average for the systems (exponential of the average of the natural logs of the residual norms) is $5.3\times 10^{-6}$.

\begin{table}[t!]
\caption{Example~8:  BWM2000 matrix of size $n = 2000$ with 10 right-hand sides.  Compare a deflated polynomial for multiple right-hand sides to BiCGStab.}

\vspace*{-3pt}
\begin{center}
\begin{tabular}{|c|c|c|c|} \hline\hline
method                  & total MVP's   & total time                & log avg.\thinspace res.\thinspace norm  \\ \hline \hline
Deflated Double Poly    & \smash{\raisebox{-6pt}{10{,}460}}      & 0.37 + 0.09 + 0.07        & \smash{\raisebox{-6pt}{$3.1\times 10^{-10}$}}          \\
deg = $50\times 16-1 = 799$     &               &  = 0.53 seconds           &                       \\ \hline
BiCGStab                & \smash{\raisebox{0pt}{550{,}313}}     & 6.9 seconds              &  \smash{\raisebox{0pt}{$5.3\times10^{-6}$}}          \\ \hline
\hline
\end{tabular}
\end{center}
\label{Tab:BWM}
\end{table}

{\it Example 9.}  Consider the one-dimensional convection-diffusion equation $- u{''} + \alpha u'- 30^2 u = f $.  This example shows that it can take careful implementation of the deflated polynomial for significantly nonnormal and indefinite problems, but the method can be cheap and accurate.  
Table~\ref{Tab:CD1} shows results for $\alpha=0,5,25$.  With $n=1000$, the matrix is fairly ill-conditioned and the $-30^2 u$ term makes these matrices indefinite.  For the symmetric case of $\alpha=0$, the first system is solved with PP(25)-GMRES to relative residual tolerance of $10^{-11}$ (taken to this level to give eigenvectors a chance to converge further), requiring 64~outer iterations and 1625 matrix-vector products. Then 30~eigenvectors are approximated.  Next, the deflated polynomial is found of degree~249, and it accurately solves the other right-hand sides.  The residual norms are $6.9\times 10^{-9}$ or better, and the average number of matrix-vector products (including the first right-hand side) is~387.   BiCGStab needs over two orders of magnitude more matrix-vector products, but it does produce accurate results.   

\begin{table}[b!]
\caption{Example 9: one-dimensional convection-diffusion equation $- u{''} + \alpha u' - 30^2 u = f $ with $n=1000$.   }

\vspace*{-3pt}
\begin{center}
\begin{tabular}{|c|c|c|c||c|c|c|}  \hline\hline
& \multicolumn{3}{|c||}{Double Polynomial}                  &  \multicolumn{2}{|c|}{BiCGStab }              \\  \hline
$\alpha$ & rtol's   & degree of     & average               & average       & log avg.             \\ 
        & $(\log_{10})$ & deflated polynomial & MVP's                 & MVP's         & res.\thinspace norm         \\ \hline \hline
0       & $-11$,$-9$,$-8$    & $25\times 10-1=249$   & 392                   & 57{,}868      & $6.6\times 10^{-9}$         \\ \hline
5       & $-11$,$-9$,$-8$    & $25\times 62-1=1549$  & 1562                  & \smash{\raisebox{-6pt}{61{,}623}}      & \smash{\raisebox{-6pt}{$2.9\times 10^{-7}$}}           \\ \cline{1-4}
5       & $-11$,$-3$,$-8$   & $25\times 4-1=99$     & 437                   &  &  \\ \hline
25      & $-10$,$-2$,$-8$    & $25\times 3-1=74$     & 542                   & 66{,}275      & $1.4\times 10^{-3}$           \\ \hline
\hline
\end{tabular}
\end{center}
\label{Tab:CD1}
\end{table}

Next, consider $\alpha>0$.  BiCGStab requires slightly more matrix-vector products, and is less accurate.  The deflated polynomial is not effective for $\alpha = 5$: it has about the same degree as if no deflation was used, due to inaccuracy of the deflation in removing eigenvectors.  PP-GMRES converges rapidly at first due to the partial deflation, but slows down once the eigenvector components of the residual vector are reduced down to the level of the deflated components. At that point, the Krylov method must deal with these small eigenvalues, and the convergence plateaus while that happens;  
see the solid (red) curve in Figure~\ref{Fig:pAif21}. 
For effective use of the deflation, it pays to stop the PP-GMRES method when the not fully deflated eigenvalues begin to impede convergence. 
Doing so gives a low degree deflated polynomial that needs to be used more than once, with a deflation in-between each application of the polynomial.  
Table~\ref{Tab:CD1} shows the results of running PP(25)-GMRES on the second system to $\rtol_2 = 10^{-3}$, then applying it three times (with projections) to $b^{(j)}$ for $j=3,\ldots,10$ and twice for the partly solved right-hand side for $j=2$.   This approach gives accurate results, with a log average residual norm for the nine right-hand sides of $5.8\times 10^{-10}$.  The average matrix-vector products for all 10 systems is only 427.  

For higher nonnormality, $\alpha = 25$, a good result comes from  $\rtol_2 = 10^{-2}$.  The deflated polynomial is applied five or six times for $j=2,\ldots, 10$ to reach the desired residual norm level of $10^{-8}$.  While this process uses two orders of magnitude fewer matrix-vector products than BiCGStab and gives better accuracy, it required experimentation.  An automated method for choosing $\rtol_2$  would be quite desirable.

\begin{figure}[t!]
\begin{center}
\includegraphics[height=2.0in]{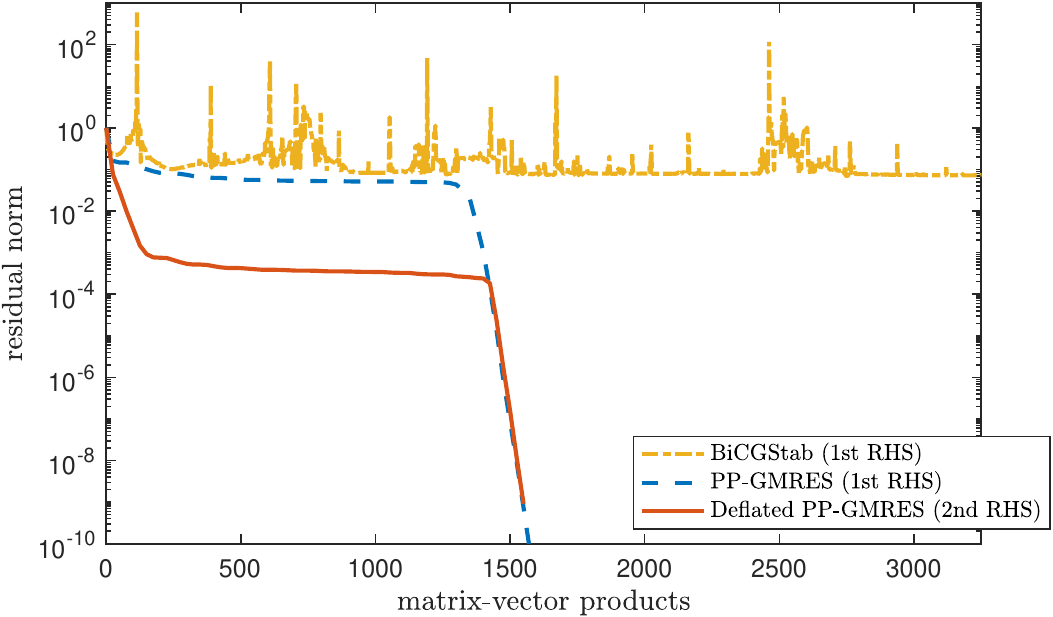}
\end{center}
\vspace{-7pt}
\caption{Example~9: convection-diffusion equation $- u{''} + 5 u'- 30^2 u = f $, $n=1000$.  Convergence is shown for BiCGStab, PP(25)-GMRES with no restarting, and for deflated PP(25)-GMRES.}
\label{Fig:pAif21}
\end{figure}

{\it Example 10.} In Lattice Quantum Chromodynamics (QCD), one important task is to estimate the trace of the inverse of a large non-Hermitian complex matrix.  This can be done with a Monte Carlo approach called Hutchinson's method~\cite{HutchTrace}, which requires solving many linear systems with random right-hand sides.  Here we use 10  right-hand sides with entries randomly drawn from $\{\pm1, \pm{\rm i}\}$.   The QCD matrix is from a $16^3$-by-$24$ lattice, giving $n = 1{,}179{,}648$.  The desired relative residual tolerance for the linear systems is $\rtol_3 = 10^{-5}$.  
The spectrum is roughly inside a circle in the right-half of the complex plane, with eigenvalues approaching the origin from above and below; the presence of small eigenvalues makes deflation important.  A two-sided projection is needed.  Once right eigenvectors are known, the left eigenvectors can be computed easily by exploiting the structure of this application.  Thus only one deflation projection is needed despite the significant nonnormality.  This projection solves $W^*\kern-1pt AV d = W^*r_0$ then $\widehat{x} = V d$, where $W$ has columns (approximately) spanning the left eigenvectors, and $V$ has columns (approximately) spanning the right eigenvectors.  

Table~\ref{Tab:QCD} compares the deflated polynomial to the non-deflated polynomial and BiCGStab.   The non-deflated polynomial is found by solving the first right-hand side with PP(40)-GMRES to relative residual norm below $10^{-6}$.  This requires 52 iterations, so the polynomial is of high degree ($40\times 52-1 = 2081$).   This polynomial solves the next nine systems to $2.4\times 10^{-5}$ or better, just worse than the desired accuracy; a slightly higher degree polynomial is needed.  BiCGStab is solved to relative residual norm of $10^{-5}$ and gives accurate answers, but requires much more time.  

Finally, to find the deflated polynomial, the first right-hand side is solved with PP(40)-GMRES to relative residual tolerance of $\rtol_1 = 10^{-12}$, which  takes 78~iterations.  Then approximate eigenvectors are computed, and here only the $\nev=34$ of them that have residual norm below $10^{-3}$ are used (along with the associated left eigenvectors).  Deflated PP(40)-GMRES for the second right-hand side with $\rtol_2=10^{-6}$ runs for only 14~iterations and generates a polynomial of degree $40\times 14-1 = 559$, much lower than the non-deflated case.  This polynomial solves the remaining eight systems to relative residual norms of $3.4\times 10^{-6}$ or better.  Overall, the time is reduced by more than a factor of two compared to the non-deflated polynomial.  More time is needed for the first right-hand side (7.1 minutes versus 4.1 minutes), since it is over-solved.  However, the next nine are faster (9.6 minutes instead of 31.7 minutes).  

\begin{table}[t!]
\caption{Example~10: QCD matrix of size $n = 1{,}179{,}648$ with 10 random right-hand sides.  
Comparison of a deflated polynomial, non-deflated polynomial, and the conventional BiCGstab algorithm.}
\begin{center}
\begin{tabular}{|c|c|c|c|} \hline\hline
method                  & total MVP's   & total time             \\ \hline \hline
Deflated Double Polynomial    & \smash{\raisebox{-6pt}{8869}}      & 7.1 + 1.1 + 8.5     \\
deg = $40\times14-1 = 559$     &               &  = 16.7 minutes        \\ \hline
Non-deflated Double Polynomial & \smash{\raisebox{-6pt}{20{,}898}}     & 4.1 + 31.7     \\
deg = $40\times 52-1 = 2079$    &               &  = 35.8 minutes        \\ \hline
BiCGStab                & 203{,}481     & 375 minutes            \\ \hline
\hline
\end{tabular}
\end{center}
\label{Tab:QCD}
\end{table}

\section{Stability of the polynomial} \label{sec:stab}

When $A$ has a few eigenvalues that stand out from the rest of the spectrum, polynomials from GMRES can be prone to instability, especially as the degree increases.
Stability can often be improved by adding extra copies of the harmonic Ritz values near these outstanding eigenvalues; see Subsection~\ref{ssec:rev2}.  
We briefly consider implications for polynomials approximating $A^{-1}$.  The next example shows that it can be possible to find an effective polynomial even when there are very outstanding eigenvalues.

{\it Example 11.}
We use bidiagonal matrices with increasing separation of the large eigenvalues and thus increasing difficulty for stability.  All matrices are size $n=2500$ and have $0.2$ on the superdiagonal (to give a nontrivial departure from normality). Matrix~1 has eigenvalues $1, 2, 3, \ldots, n$ on the main diagonal.  Matrix~2 has diagonal $0.1, 0.2, 0.3, \ldots, 0.9, 1, 2, 3, \ldots, 2490, 2491$.  
While the large eigenvalues are not especially separated, they are more separated relative to the size of the small eigenvalues than for Matrix~1.  
Matrix~3 has diagonal $0.1, 0.2, 0.3, \ldots, 0.9, 1, 2, 3, \ldots, 2490, 2600$, giving one very well separated eigenvalue.  Matrix~4 has five outstanding eigenvalues, with diagonal  $0.1, 0.2, 0.3, \ldots, 0.9, 1$, $2, 3, \ldots, 2486, 2600, 2700, 2800, 2900, 3000$.

Table~\ref{Tab:Stab1} shows results both without and with the stability control of adding roots from~\cite[alg.~2]{PPGStable}.  The first right-hand side is solved with (full) GMRES to relative residual norm of $10^{-11}$, and the polynomial $p$ is then applied to nine other right-hand sides.  Without adding roots, the polynomial is effective in solving multiple right-hand sides only for Matrix 1.  In that case, the nine extra systems all have residual norms below $3.1\times 10^{-11}$.
For the next three matrices, $p$ is increasingly unstable with high $\pof$ values (as defined in~\cite[alg.~2]{PPGStable}) that indicate $p$ has a steep slopes; indeed, $p$ provides inaccurate solutions for the nine additional systems.

With stability control from~\cite[alg.~2]{PPGStable} using $\pofcutoff = 8$, results are good even with outlying eigenvalues and many added roots.  Matrix 3 gives the least accurate results.  However if $\pofcutoff = 4$, then all nine residual norms are below $2.3\times 10^{-11}$.  

The improvement with stability control is remarkable.  However, the next example is designed to challenge the stability control procedure.

\begin{table}[t!]
\caption{Example~11:  bidiagonal matrices with some increasingly outlying eigenvalues.  The effect of stability control is shown for solving systems with multiple right-hand sides.}
\begin{center}
\begin{tabular}{|c|c|c|c||c|c|}  \hline \hline
& \multicolumn{3}{|c||}{Without Stability Control}              &  \multicolumn{2}{|c|}{With Stability Control }  \\  \hline
\smash{\raisebox{-6pt}{matrix}} & degree of     & \smash{\raisebox{-6pt}{max $\pof$}}        & max residual      & roots      & max residual          \\[-2pt]
        & polynomial    &                   & norm              & added             & norm              \\ \hline \hline
1       & 324           & $2.5\times10^{1}$      & $3.1\times10^{-11}$    & 0                 & $3.1\times10^{-11}$    \\ \hline
2       & 596           & $9.5\times10^{22}$     & $5.4\times10^{6}$      & 12                & $2.7\times10^{-11}$    \\ \hline
3       & 596           & $2.3\times10^{103}$    & $5.3\times10^{87}$     & 19                & $5.7\times10^{-9}$     \\ \hline
4       & 597           & $7.9\times10^{216}$    & $4.5\times10^{201}$    & 68                & $1.5\times 10^{-11}$    \\ \hline
\hline
\end{tabular}
\end{center}
\label{Tab:Stab1}
\end{table}

{\it Example 12.}  We choose a diagonal matrix with entries $0.1, 0.2, 0.3, \ldots, 1, 2, 3, \ldots$, $50, 551, 552, 553, \ldots 1000,1501, 1502, 1503, \ldots, 2000,2501, 2502, 2503, \ldots, 3000, 3501$, \\ $3502, 3503, \ldots, 4491$.  This spectrum has four large gaps, and the residual polynomial $\pi$ has steep slope at eigenvalues at the gap edges.  Also contributing to the steepness is the presence of enough small eigenvalues to push the polynomial degree high.  Slopes are especially steep for the eigenvalue at~50 and the ones just below it.   

\begin{figure}[b!]
\begin{center}
\includegraphics[width=4in]{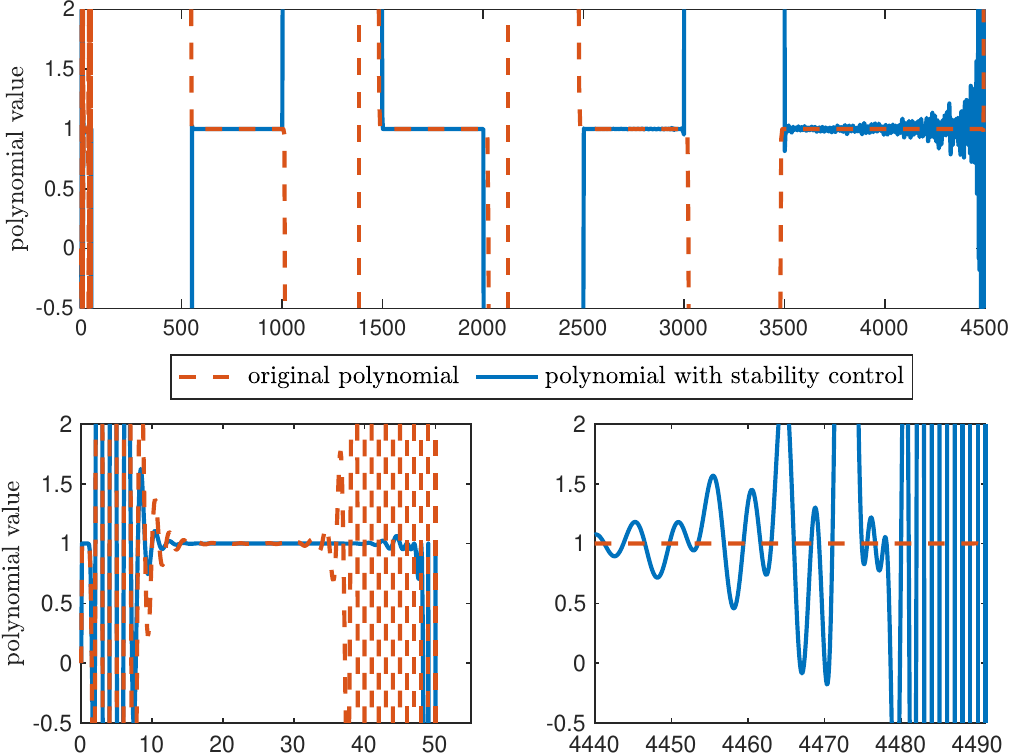}
\end{center}
\vspace{-5pt}
\caption{Example~12:  a matrix with four gaps in the spectrum, $n=2500$.  The top plot shows the original $\phi$ polynomial of degree $455$ (dashed red line) and with extra roots added for stability control (solid blue line).  The bottom plot zooms around the small and large parts of the spectrum.}
\label{fig:GapExa}
\end{figure}

We solve the first system to relative residual norm below $10^{-11}$, generating $p$ of degree~454.  Without  stability control, this polynomial does not accurately solve other systems:  the residual norms for nine right-hand sides are as high as $9.4\times 10^{-3}$.  With stability control, the results are even \emph{worse}: the largest of the nine residual norms is $3.2\times 10^{-1}$.  Seven roots are added, including extra roots near eigenvalues 50, 49, 48, 47 and 46, along with two at harmonic Ritz values in gaps of the spectrum.  
Figure~\ref{fig:GapExa} (top) shows the polynomial $\phi$ before  roots are added (dashed line) and with roots (solid line).  This polynomial needs to be close to~1 over the spectrum for the method to be effective.  Figure~\ref{fig:GapExa} (bottom) examines these polynomials over the small and large eigenvalues. 
Adding five roots near 50 multiplies the $\pi$ polynomial by linear factors like $(1 - z/50)$,  each of which is near $-100$ at the large eigenvalues.  While the polynomial oscillates less on the small part of the spectrum, it has big amplitudes around the large eigenvalues.  This causes two problems: instability due to a $\pof$ of $5\times 10^{13}$ at the largest eigenvalue, and, more importantly, the polynomial $\phi$ varies significantly from~1 at quite a few of the eigenvalues (e.g.,  4472, 4473 and 4474).  This inaccuracy makes $p(A)$ useless for solving linear equations.  

We propose an improved version of Algorithm~\cite[alg.~2]{PPGStable} that performs the stability test on roots in increasing order of magnitude, and then updates the $\pof$ values when a root is added; see Algorithm~\ref{Alg:RootAdding2}.  This algorithm improves the results so that the maximum of the nine residual norms is $4.0\times 10^{-6}$.  The limited accuracy is not due to instability; rather, the $\phi$ polynomial is still not close enough to~1 at some eigenvalues.  The largest deviation is at $\lambda_{1512}=3503$ where $\phi(\lambda_{1512}) = 1-5.2\times 10^{-5}$, which is significantly far from~1.  
The $\phi$ polynomial is more accurate at larger eigenvalues due to two added roots at the largest two eigenvalues, a result of the updated $\pof$ values.  This example demonstrates the danger of adding roots to the GMRES polynomial, because then the polynomials no longer correspond to a minimum residual method and may not be accurate at all eigenvalues. However, Algorithm~\ref{Alg:RootAdding2} with $\pof$ updating did improve accuracy.  In other testing with matrices that are difficult due to gaps in the spectrum or due to indefiniteness, Algorithm~\ref{Alg:RootAdding2} often improves the situation.

\begin{figure}[t!]
\begin{center}
\includegraphics[width=3.25in]{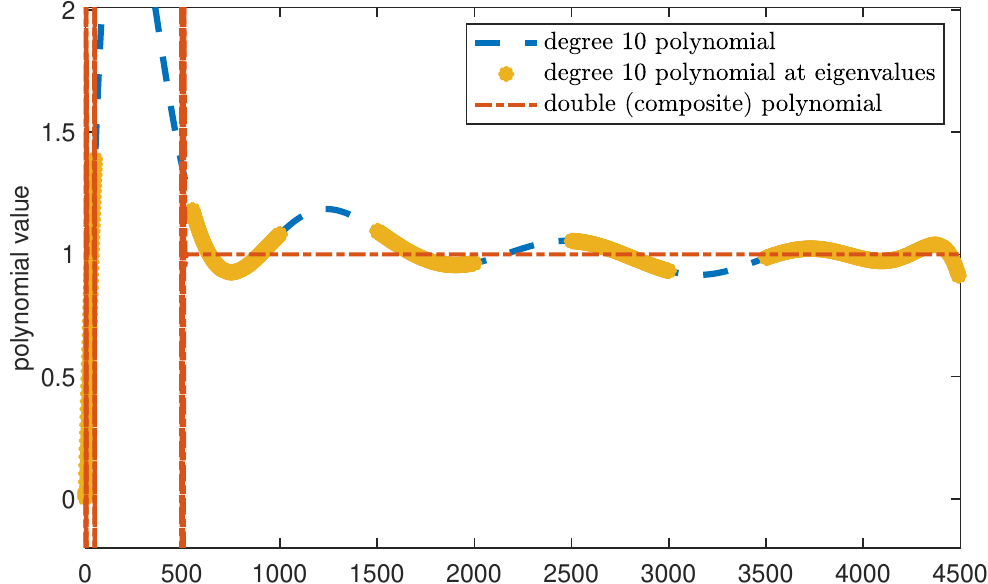}
\end{center}
\vspace*{-5pt}
\caption{Example~12:  a matrix with four gaps in the spectrum, $n=2500$, double polynomial approach.  The inner polynomial $\phi_{\rm in}$ is highlighted at the eigenvalues (blue line, yellow dots).  Compare this to the double polynomial $\phi_{\rm out}(\phi_{\rm in}(z)) = z*p_{\rm in}(z)*p_{\rm out}(\phi_{\rm in}(z))$ of degree~$670$ (red line).}
\label{fig:GapExb}
\vspace*{-10pt}
\end{figure}

\begin{algorithm}
    \caption{Adding roots to $\pi(z)$ for stability with $\pof$ updating}
    \begin{description}
 \item[1.] {\bf Setup:} Assume the $d$ roots ($\theta_1, \ldots, \theta_d$) of $\mrpoly$ have been computed and then sorted according to the modified Leja ordering~\cite[alg.~3.1]{BaHuRe}.  (For high-degree polynomials and/or large magnitude roots, use sums of logs in place of products, to prevent overflow and underflow in the Leja ordering calculations.)
 \item[2.] {\bf Compute \boldmath ${\it pof}(k)$:} For $k=1,\ldots,d$, compute $\pof(k) = \prod_{i \neq k} |1-\theta_k/\theta_i|$.  
 \item[3.] \textbf{Add roots using a reordered list of roots and update \boldmath${\it pof}$ values:} Reorder the roots and their corresponding $\pof\kern1pt$'s, 
by increasing magnitude.  For each root $\theta_k$ starting with the smallest, compute $\lceil \log_{10}(\pof(k)) - \pofcutoff)/14\rceil$.  Add that number of $\theta_k$ copies to the list of roots with the Leja ordering.  Add the first to the end of the list; if there are others, space them evenly between the first and last occurrence of $\theta_k$ (keeping complex roots together).
 Based on the added root(s), update the $\pof$ values of all roots that have not yet been considered.  Then continue checking the $\pof$ values.
 \item[4.] {\bf Apply a second Leja ordering:} For the list of $\theta_k$ values that is Leja ordered aside from the stabilizing roots, perform a second Leja ordering.  To give distinct values for the Leja algorithm, perturb these roots slightly for determining the order, but do not change the actual $\theta_k$ values.  (We  perturb to $(1 + 10^{-12}*\mbox{\em randn})*\theta_k$, where ``{\em randn}'' is a random Normal(0,1) number.)
\end{description}
\label{Alg:RootAdding2}
\end{algorithm}

We now describe another way to improve stability through use of the double polynomial.  For the same matrix with gaps, we run PP(10)-GMRES to relative residual norm below $10^{-11}$, which takes 67~iterations; this process yields the double polynomial.  No extra roots are needed for stability for either the inner or outer components of the double polynomial.  The $p$ polynomial is degree $10\times 67-1 = 669$, higher than the polynomial from GMRES.\ \   Applying this polynomial to the nine extra right-hand sides gives residual norms at or below $2.1\times 10^{-9}$.  Figure~\ref{fig:GapExb} shows the degree~10 inner polynomial $\phi_{\rm in}$ as a dashed line and highlights the values at the eigenvalues.  It also shows the composite $\phi$ polynomial, $\phi(z) = \phi_{\rm out}(\phi_{\rm in}(z))$, which is near $1$ throughout the spectrum.  This polynomial is also near $1$ over most of the gaps in the spectrum: the inner polynomial maps the gaps (except for the first one) into a zone where eigenvalues are also mapped, and so the outer polynomial needs to be near $1$ there.  The fact that the double polynomial is near $1$ over most gaps, instead of being unconstrained, partly explains why the higher degree is needed. 

Stability control and double polynomials enable polynomial approximation to the inverse for some difficult matrices, but these approaches are not foolproof.  One can construct adversarial examples with enough small eigenvalues and gaps in the spectrum to make an effective polynomial quite hard to find.  On the other hand, many applications have matrices that lack such challenges.  Polynomial approximation can be effective for these problems, even when a high-degree polynomial is needed.

\section{Conclusion}
It is often possible to construct an accurate moderate-degree polynomial approximation to the inverse of a large matrix.   Theory shows that accuracy of the polynomial generally follows the GMRES residual, although it can be effected by nonnormality and the interplay of the right-hand side and the eigenvectors.  The approximating polynomial can be found in several ways; here we focused on full GMRES and polynomial preconditioned GMRES, the latter of which builds double (composite) polynomials.  This composite approach can be more efficient due to reduced orthogonalization costs, though higher degree polynomials are often needed.

Applications for approximate polynomial inverses include solving systems of linear equations with multiple right-hand sides.  For difficult problems, the polynomial is often better than BiCGStab in expense and accuracy.  However, like all Krylov methods, a polynomial is not always effective.  This is particularly true for matrices that are ill-conditioned and have outstanding eigenvalues.  Our stability control often helps, as does using a double polynomial.  Further refinement of the stability control procedure is a promising area for further research. 

A deflated version of the polynomial uses approximate eigenvectors to lower the degree.  Multiple applications of the deflated polynomial may be needed, particularly for significantly nonnormal matrices.  Some fine-tuning is needed. This deflation can reduce the degree of the GMRES polynomial required for convergence, hence promoting stability.

In future research, we plan to consider both the nonsymmetric and symmetric Lanczos algorithms.  These methods have some natural stability control through roundoff error, which produce extra copies of Ritz values corresponding to outstanding eigenvalues (so-called ``ghost'' Ritz values).  Nonsymmetric Lanczos can find both right and left eigenvectors, which could help construct polynomials for deflation; the stability of this process will require careful assessment.  
   
\section*{Acknowledgments} The authors would like to thank the referees for their suggestions to improve the paper.

\end{document}